\documentclass[12pt,a4paper,leqno]{amsart}

\title[Representations of a.p. pseudodifferential operators]{Representations of almost periodic pseudodifferential operators and applications in spectral theory}

\author[P. Wahlberg]{Patrik Wahlberg}

\address{Dipartimento di Matematica, Universit{\`a} di Torino, Via Carlo Alberto 10, 10123 Torino (TO), Italy.}

\email{patrik.wahlberg@unito.it}

\usepackage{latexsym}
\usepackage{amsmath}
\usepackage{amssymb}
\usepackage{amsthm}
\usepackage{amsfonts}
\usepackage{mathrsfs}
\usepackage{calc}
\usepackage{cite}

\numberwithin{equation}{section}

\newtheorem{thm}{Theorem}
\numberwithin{thm}{section}

\newcommand{\rubrik}{}
\newtheorem{prop}[thm]{Proposition}
\newtheorem{cor}[thm]{Corollary}
\newtheorem{lem}[thm]{Lemma}

\theoremstyle{definition}

\newtheorem{defn}[thm]{Definition}

\theoremstyle{remark}

\newtheorem{rem}[thm]{Remark}

\newcommand{\qo}{\mathbb Q}
\newcommand{\ro}{\mathbb R}
\newcommand{\rr}[1]{\mathbb R^{#1}}
\newcommand{\zz}[1]{\mathbb Z^{#1}}
\newcommand{\rrb}[1]{\mathbb R_B^{#1}}
\newcommand{\nn}[1]{\mathbb N^{#1}}

\newcommand{\co}{\mathbb C}
\newcommand{\no}{\mathbb N}

\newcommand{\pddm}[2]{\partial_#1^#2}

\newcommand{\Lop}{\mathscr L}

\newcommand{\ep}{\varepsilon}

\newcommand{\la}{\lambda}
\newcommand{\La}{\Lambda}
\newcommand{\Ker}{{\rm Ker}}
\newcommand{\Coker}{{\rm Coker}}
\newcommand{\Ran}{{\rm Ran}}
\newcommand{\Dom}{{\rm Dom}}

\newcommand{\wpr}{{\text{\footnotesize $\#$}}}

\newcommand{\eabs}[1]{\langle #1\rangle}

\newcommand{\vrum}{\vspace{0.1cm}}

\newcommand{\wt}{\widetilde}
\newcommand{\wh}{\widehat}
\newcommand{\re}{\mbox{Re}}
\newcommand{\im}{\mbox{Im}}
\newcommand{\linspan}{\mbox{span}}

\begin{document}

\begin{abstract}
The paper concerns algebras of almost periodic pseudodifferential operators on $\rr d$ with symbols in H\"ormander classes.
We study three representations of such algebras, one of which was introduced by Coburn, Moyer and Singer and the other two inspired by results in probability theory by Gladyshev.
Two of the representations are shown to be unitarily equivalent for nonpositive order.
We apply the results to spectral theory for almost periodic pseudodifferential operators acting on $L^2$ and on the Besicovitch Hilbert space of almost periodic functions.
\end{abstract}


\maketitle

\let\thefootnote\relax\footnotetext{Author's address: Dipartimento di Matematica, Universit{\`a} di Torino, Via Carlo Alberto 10, 10123 Torino (TO), Italy. Phone +390116702944, fax +390116702878. Email patrik.wahlberg@unito.it.}

\section{Introduction}

The paper is a study of three representations of almost periodic (a.p.) pseudodifferential operators on $\rr d$ with applications in the spectral theory of such operators. The symbols $a \in C^\infty(\rr {2d})$ we study belong to H\"ormander classes $S_{\rho,\delta}^m$ with $0 < \rho \leq 1$, $0 \leq \delta < 1$, $\delta
\leq \rho$ and $m \in \ro$. We are interested in the case when $a(\cdot,\xi)$ is a continuous a.p. function in the sense of Bohr for all $\xi \in \rr d$, and study the corresponding algebra of pseudodifferential operators in the Kohn--Nirenberg quantization, denoted $APL_{\rho,\delta}^\infty$. Such operators are called a.p. pseudodifferential operators, and have been investigated by Coburn, Moyer and Singer \cite{Coburn1}, Dedik \cite{Dedik1}, Filippov \cite{Filippov1}, 
Oliaro, Rodino and Wahlberg \cite{Oliaro1}, 
Pankov \cite{Pankov1}, Rabinovich \cite{Rabinovich1}, Shubin \cite{Shubin1,Shubin2,Shubin3,Shubin4,Shubin4b}, and others.
Recently, Ruzhansky and Turunen \cite{Ruzhansky1} have studied the related class of pseudodifferential operators on the torus $\mathbb T^d$ where the almost periodicity is replaced by periodicity.
Pseudodifferential operators on more general compact Lie groups are studied in their monograph \cite{Ruzhansky2}.

In an earlier paper \cite{Wahlberg1} we have shown that the transformation of a symbol $a$, defined by
\begin{equation}\nonumber
U(a)(\xi)_{\lambda,\lambda'} = \mathscr M_x ( a(x,\xi-\lambda') e^{- 2 \pi i
x \cdot (\lambda'-\lambda)} ), \quad \la, \la', \xi \in \rr d
\end{equation}
($\mathscr M_x$ denotes the mean value functional of a.p. functions), introduced by Gladyshev \cite{Gladyshev1} in probability theory, may be used as a representation of $APL_{\rho,\delta}^\infty$. More precisely, for fixed $\xi \in \rr d$, $U(a)(\xi)$ is considered an operator on $l^2(\rr d)$ with kernel $U(a)(\xi)_{\lambda,\lambda'}$, and $\xi \mapsto U(a)(\xi)$ is used as an operator-valued symbol of a Fourier multiplier operator $U(a)(D)$ acting on vector-valued function spaces like $\mathscr S(\rr d, l^2(\rr d))$. It is shown in \cite{Wahlberg1} that the map $APL_{\rho,\delta}^\infty \ni a(x,D) \mapsto U(a)(D)$ preserves identity, composition, formal adjoint, and positivity (the latter result was proved by Gladyshev \cite{Gladyshev1} for a different class of symbols), and therefore deserves the designation of a representation.

Moreover, the evaluation at the origin of the symbol $U(a)(0)$ gives a representation of $APL_{\rho,\delta}^\infty$
on the space of unbounded operators on $l^2(\rr d)$, such that each $a(x,D) \in APL_{\rho,\delta}^m$ is mapped to a
continuous operator $l_s^2(\rr d) \mapsto l_{s-m}^2(\rr d)$, $s \in \ro$. The sequence space $l_s^2(\rr d)$ consists of functions on $\rr d$ which are zero everywhere except for a countable set, with norm
\begin{equation}\nonumber
\| x \|_{l_s^2} = \left( \sum_{\la \in \rr d} | x_\la|^2 (1+|\la|^2)^{s} \right)^{1/2}.
\end{equation}
Thus two representations of $APL_{\rho,\delta}^\infty$ are defined.

A third representation was introduced by Coburn, Moyer and Singer \cite{Coburn1} and studied further by Shubin \cite{Shubin2,Shubin4b}. The idea is to define the spatially translated symbol $a_x(y,\xi)=a(x+y,\xi)$ and for fixed $x \in \rr d$ let the operator $a_x(y,D)$ act on a function of two variables $u(x,y)$. That is, we set $(Au)(x,y) = a_x(y,D)u(x,\cdot)(y)$, which defines an operator acting on $B^2(\rr d) \otimes L^2(\rr d)$ where $B^2(\rr d)$ denotes the space of Besicovitch a.p. functions. In \cite{Shubin2} it is shown that $A: B^2(\rr d) \otimes H^s(\rr d) \mapsto B^2(\rr d) \otimes H^{s-m}(\rr d)$ continuously for $s \in \ro$, where $H^s(\rr d)$ denotes a Sobolev space on $\rr d$. Moreover, it shown that $a(y,D) \mapsto A=A(a(y,D))$ preserves identity, composition and formal adjoint, with rather brief arguments.

An important feature of the representation $A(a(y,D))$ is the fact that it is adjoined to a von Neumann algebra $\mathscr A_B$ which is a factor of type II$_\infty$ \cite{Coburn1,Shubin2,Shubin4b}. If $m \leq 0$ and $a \in APS_{\rho,\delta}^m$ this means that $A(a(y,D)) \in \mathscr A_B$. Based on this fact, Shubin \cite{Shubin2} develops
results on the spectral asymptotics of linear PDEs with smooth a.p. coefficients, that are uniformly elliptic and essentially selfadjoint.
Index theory in this context is discussed in \cite{Coburn1,Shubin4b}.

This paper contains four contributions.

First we extend the results of \cite{Wahlberg1} and simplify their proofs. In particular, we prove that for $a \in APS_{\rho,\delta}^m$, the operator
\begin{equation}\nonumber
U(a)(D): H^{s}(\rr d, l_{s}^2(\rr d)) \mapsto H^{s-|s|-|m-s|}(\rr d, l_{s-m}^2(\rr d)), \quad s \in \ro,
\end{equation}
is continuous. Here we denote by $H^{s}(\rr d, l_{s}^2(\rr d))$ a vector-valued Sobolev space. If $m \leq 0$ we thus have $U(a)(D) \in \Lop(L^2(\rr d,l^2(\rr d))$. The map $a(y,D) \mapsto U(a)(D)$ is a faithful $^+$-representation of $APL_{\rho,\delta}^\infty$ on
an algebra of unbounded operators on $L^2(\rr d,l^2(\rr d))$,
that preserves positivity in the sense that $a(y,D) \geq 0$ on the trigonometric polynomials if and only if $U(a)(D) \geq 0$.
The map $a(y,D) \mapsto U(a)(0)$ is a faithful $^+$-representation of $APL_{\rho,\delta}^\infty$ on
an algebra of unbounded operators on $l^2(\rr d)$.
It has the property that $a(y,D): H^s(\rrb d) \mapsto H^{s-m}(\rrb d)$ and $U(a)(0): l_s^2(\rr d) \mapsto l_{s-m}^2(\rr d)$ are unitarily equivalent for each $s \in \ro$. Here $H^s(\rrb d)$ denotes a Sobolev--Besicovitch space of a.p. functions (cf. \cite{Shubin1,Shubin2,Shubin4}).

Secondly we give a detailed exposition of the results by Coburn--Moyer--Singer and Shubin concerning the representation
\begin{equation}\label{representation0}
a(y,D) \mapsto A=A(a(y,D)).
\end{equation}
We prove that the operator $A$ extends to a continuous map $A: B^2(\rr d) \otimes H^s(\rr d) \mapsto B^2(\rr d) \otimes H^{s-m}(\rr d)$ for $s \in \ro$, and that the map \eqref{representation0} is a faithful $^+$-representation of $APL_{\rho,\delta}^\infty$ on an algebra of unbounded operators on $B^2(\rr d) \otimes L^2(\rr d)$, that preserves positivity in the sense that $a(y,D) \geq 0$ on the trigonometric polynomials if and only if $A \geq 0$.

Thirdly, we prove that the representations \eqref{representation0} and $a(y,D) \mapsto U(a)(D)$ are unitarily equivalent for $m \leq 0$. Under this assumption $A(a(y,D)) \in \Lop(B^2(\rr d) \otimes L^2(\rr d))$ and $U(a)(D) \in \Lop(L^2(\rr d,\l^2(\rr d)))$.

Finally, as a fourth topic we discuss applications of the representations in spectral theory of a.p. pseudodifferential operators.
Shubin has shown that the spectra of a.p. pseudodifferential operators considered as operators on $L^2(\rr d)$ and on $B^2(\rr d)$, respectively, coincide. More precisely we have $\sigma_{L^2}(\overline{a(y,D)}) = \sigma_{B^2}(\overline{a(y,D)})$ (where $\overline{a(y,D)}$ denotes the closure of an unbounded closable operator \cite{Reed1}) when
$0 \leq \delta < \rho \leq 1$ and $a \in APS_{\rho,\delta}^0$ or $a \in APHS_{\rho,\delta}^{m, m_0}$ with $m_0>0$ \cite{Shubin3}. (Here $APHS_{\rho,\delta}^{m, m_0}$ denotes formally hypoelliptic symbols, see Definition \ref{aphypo1}.)
He has also shown that $\sigma_{L^2}(\overline{a(y,D)}) = \sigma_{L^2,\rm ess}(\overline{a(y,D)})$, i.e. the whole spectrum  is essential under the same assumptions \cite{Shubin2}.
In \cite{Rozenblum1}, Rozenblum, Shubin and Solomyak indicate with a brief sketch that the same assumptions imply $\sigma_{B^2}(\overline{a(y,D)}) = \sigma_{B^2,\rm ess}(\overline{a(y,D)})$. We give a detailed proof of this result.
As a consequence of these results on the essential spectrum, it follows that an operator in $APL_{\rho,\delta}^0$ (where $\rho$,$\delta$ satisfy \eqref{rhodeltakrav1}) cannot be compact on neither $L^2$ nor on $B^2$ unless it is zero.

If $a \in APS_{\rho,\delta}^m$ and the symbol $a$ does not depend on $\xi$, then we show that $\sigma_{B^2}(a(y,D)) = \sigma_{B^2,\rm ess}(a(y,D)) = \overline{\Ran (a)}$. If $a$ does not depend on $x$, then $\sigma_{B^2}(\overline{a(D)}) = \sigma_{B^2,\rm ess}(\overline{a(D)}) = \overline{\Ran (a)}$, the point spectrum is $\sigma_{B^2,\rm p}(\overline{a(D)}) = \Ran(a)$ and the continuous spectrum is $\sigma_{B^2,\rm cont}(\overline{a(D)}) = \overline{\Ran (a)} \setminus \Ran(a)$.
If there exists $\xi_0 \in \rr d$ such that $a(x,\xi_0)=\mathscr M (a(\cdot,\xi_0))$ for all $x \in \rr d$, then
$\mathscr M (a(\cdot,\xi_0)) \in \sigma_{B^2,\rm ess}(\overline{a(x,D)})$.

Finally, we have the following invariances of the spectrum over the three representations. If
$0 < \rho \leq 1$, $0 \leq \delta < 1$, $\delta \leq \rho$ and $a \in APS_{\rho,\delta}^0$ then
\begin{equation}\nonumber
\sigma_{L^2}(a(y,D)) = \sigma_{B^2 \otimes L^2}(A(a(y,D)) = \sigma_{l^2}(U(a)(0)) = \sigma_{L^2(\rr d, l^2)}(U(a)(D)).
\end{equation}
(The first identity is proved in \cite{Shubin4b}).
If $0 \leq \delta < \rho \leq 1$ and $a \in APHS_{\rho,\delta}^{m,m_0}$ where $m_0>0$ then
\begin{equation}\nonumber
\sigma_{L^2}(\overline{a(x,D)}) = \sigma_{B^2 \otimes L^2}(\overline{A(a(y,D))}) = \sigma_{l^2}(\overline{U(a)(0)}).
\end{equation}
(The first identity is proved in \cite{Shubin2,Shubin4b} under slightly less general hypotheses.) Note that $\sigma_{L^2(\rr d, l^2)}(\overline{U(a)(D)})$ is missing in the latter sequence of equalities. Although $U(a)(D)$ is adjoined to a von Neumann algebra which is a factor of type II$_\infty$, in the same way as $A(a(y,D))$, the latter representation is not as useful for spectral theory as the latter, since there is no connection to the spectrum of the original operator, that is $\sigma_{L^2}(\overline{a(x,D)}) = \sigma_{B^2}(\overline{a(x,D)})$.

\section{Preliminaries}

$C$ denotes a positive constant that may vary over
equalities and inequalities.
For an integer $m \geq 0$ we denote by $C^m (\rr d)$ the space of
functions such that $\partial^\alpha f$ is continuous for
$|\alpha|\leq m$, and $C^\infty = \bigcap_m  C^m$.
$C_b (\rr d)$ is the space of continuous and supremum bounded functions,
and $C_b^\infty (\rr d)$ is the space of smooth functions such that a derivative of any order is supremum bounded.
$C_c^\infty(\rr d)$ denotes the space of compactly
supported smooth (test) functions.
The Schwartz space of smooth rapidly decreasing functions is written
$\mathscr S (\rr d)$ and its dual $\mathscr S' (\rr d)$ is the space of tempered distributions.
A character on $\rr d$ with frequency $\xi \in \rr d$ is denoted $e_\xi(x) = e^{2 \pi i \xi \cdot x}$, and
$TP(\rr d)$ is the space of trigonometric polynomials on
$\rr d$ consisting of functions of the form
\begin{equation}\label{trigpol1}
f(x) = \sum_{n=1}^N a_n e_{\xi_n} (x), \quad a_n \in
\co, \quad \xi_n \in \rr d.
\end{equation}
If $X$ is a topological vector space and the coefficients $a_n \in X$ in \eqref{trigpol1} then $f \in TP(\rr d,X)$ is an $X$-valued trigonometric polynomial.

As usual we denote by $\delta_a=\delta_0(\cdot-a)$, $a \in \rr d$, a translated Dirac measure at the origin of $\rr d$. We denote by $\delta_{(a)}$ the function on $\rr d$ that is zero everywhere except $a \in \rr d$ where it equals one.
The same notation is used when the domain is $\zz d$, in which case $\delta_{(0)}$ is the Kronecker delta.
Translation is written $T_x f(y)=f(y-x)$ and modulation $M_\xi f(x)=e^{2 \pi i x \cdot \xi} f(x)$ for functions defined on $\rr d$. The linear span of a finite set of vectors in a vector space is denoted $\linspan(x_1,\cdots,x_n)$.

The Fourier transform for scalar- or Hilbert space-valued functions is defined by
\begin{equation}\nonumber
\mathscr F f(\xi) = \wh f(\xi) = \int_{\rr d} f(x) e^{- 2 \pi i x \cdot \xi} dx, \quad f \in \mathscr S(\rr d).
\end{equation}

We denote coordinate reflection by $Rf(x)=f(-x)$. By $H \otimes H'$ we denote the Hilbert space tensor product \cite{Reed1} of the Hilbert spaces $H$ and $H'$. If the spaces are merely vector spaces, then we use the same notation for the algebraic tensor product, that is, the space of finite linear combinations of simple tensors.

We will use the counting measure on $\rr d$ and an associated family of weighted Hilbert spaces.
For $s \in \ro$ and ${\eabs \la}=(1+|\la|^2)^{1/2}$, $l_s^2 = l_s^2(\rr d)$ is the space of complex-valued sequences $(x_\la)_{\la \in \rr d}$ with at most countably many nonzero entries, normed by
\begin{equation}\nonumber
\| x \|_{l_s^2} = \left(\sum_{\lambda \in \rr d} \eabs{\lambda}^{2s}
|x_\lambda|^2 \right)^{1/2}.
\end{equation}
The spaces $l_s^2$ are nonseparable Hilbert spaces. When $s=0$ we write $l_0^2=l^2$,
and $l_f^2$ is the subspace of $l^2$ consisting of sequences of finite support. (This condition can also be expressed as compact support in the discrete topology of $\rr d$.)
We have
$$
\| x \|_{l_s^2} = \sup_{\| y \|_{l_{-s}^2} \leq 1} | (x,y)_{l^2} |
$$
where $(x,y)_{l^2} = \sum_{\la \in \rr d} x_\la \overline{y_\la}$.
The dual of $l_s^2$, denoted $(l_s^2)'$, can be identified isometrically $(l_s^2)'=l_{-s}^2$ by means of the bilinear form $(\cdot,\cdot)_{l^2}$.

We will use some integration theory for vector-valued functions.
A function $f: M \mapsto X$, where $(M,\mathscr B,\mu)$ is a measure space and $X$ is a Banach space, is said to be \textit{strongly measurable} \cite{Diestel1,Reed1} if there exists a sequence $(f_n)$ of simple (finite range) measurable functions such that $f_n \rightarrow f$ almost everywhere.
The function $f$ is \textit{weakly measurable} if $M \ni x \mapsto (f(x),y)$ is measurable for all $y \in X'$, which denotes the topological dual of $X$. Strongly measurable functions are always weakly measurable \cite{Reed1}, and if $X$ is a separable Hilbert space then weakly measurable functions are strongly measurable \cite[Thm.~IV.22]{Reed1}.
The Bochner integral $\int_M f(x) d\mu(x)$ (cf. \cite{Diestel1}) exists provided $f$ is strongly measurable and
$$
\int_M \| f(x) \|_X d\mu(x) < \infty.
$$

As a particular case of this situation we use functions and distributions defined on $\rr d$ that take values in $l_s^2$. More specifically we need the $l_s^2$-valued Sobolev spaces $H^t(\rr d,l_s^2)$. A tempered $l_s^2$-valued distribution $F \in \mathscr S'(\rr d,l_s^2)$ satisfies $F \in H^t(\rr d,l_s^2)$ if $\wh F \in L_{\rm loc}^2(\rr d,l_s^2)$ and
\begin{equation}\nonumber
\| F \|_{H^t(\rr d,l_s^2)} = \left( \int_{\rr d} \| \wh F (\xi) \|_{l_s^2}^2 {\eabs \xi}^{2t} d \xi \right)^{1/2} < \infty.
\end{equation}
Here $\wh F \in L_{\rm loc}^2(\rr d,l_s^2)$ refers to Bochner integrability. This implies that $\wh F$ is strongly measurable, which in turn implies that $\wh F$ takes values in a separable subspace of $l_s^2$, almost everywhere. Therefore also $F$ takes values in a separable subspace of $l_s^2$.
For a background on vector-valued distributions we refer to Amann's book \cite{Amann1}.

We will also use continuous functions $U: \rr d \mapsto \Lop(\l_s^2,l_t^2)$, where $\Lop(l_s^2,l_t^2)$ denotes the space of linear bounded operators $l_s^2 \mapsto l_t^2$ equipped with the operator norm, $s,t \in \ro$, $\Lop(l_s^2) := \Lop(l_s^2,l_s^2)$.
Since $U$ is strongly measurable it may be Bochner integrated provided its norm is integrable.

We will work with spaces of a.p. functions
(cf. \cite{Amerio1,Levitan1,Shubin4}). The basic space of uniform a.p. functions is denoted $C_{\rm ap}(\rr d)$ and defined as follows. A
set $\Omega \subseteq \rr d$ is called \textit{relatively dense} if there
exists a compact set $K \subseteq \rr d$ such that $(x+K) \cap \Omega \neq
\emptyset$ for any $x \in \rr d$. An element $\tau \in \rr d$ is
called an $\ep$-almost period of a function $f \in C_b(\rr d)$ if
$\sup_x |f(x+\tau)-f(x)|<\ep$. $C_{\rm ap}(\rr d)$ is defined as the
space of all $f \in C_b(\rr d)$ such that, for any $\ep>0$, the set
of $\ep$-almost periods of $f$ is relatively dense.
With the understanding that the uniform a.p. functions are a subspace
of $C_b(\rr d)$, this original definition by H. Bohr is equivalent
to the following three \cite{Levitan1,Shubin4}:

\begin{enumerate}

\item[(i)] The set of translations $\{ f(\cdot -x)\}_{x \in \rr
d}$ is precompact in $C_b(\rr d)$;

\vrum

\item[(ii)] $f=g \circ i_B$ where $i_B$ is the canonical
homomorphism from $\rr d$ into the Bohr compactification $\rrb d$ of
$\rr d$ and $g \in C(\rrb d)$. Thus $f$ can be extended to a
continuous function on $\rrb d$;

\vrum

\item[(iii)] $f$ is the uniform limit of trigonometric
polynomials.

\end{enumerate}

The Bohr compactification $G_B$ of a locally compact abelian group $G$ is a compact abelian topological group defined as the dual group of ($G'$, with discrete topology), that is $G_B=(G'_{\rm discr})'$ \cite{Shubin4}.

The space $C_{\rm ap}(\rr d)$ is a conjugate-invariant complex algebra of
uniformly continuous functions, and a Banach space
with respect to the $L^\infty$ norm.
For $f \in C_{\rm ap}(\rr d)$ the mean
value functional
\begin{equation}\label{meandef1}
\mathscr M(f) = \lim_{T \rightarrow +\infty} T^{-d} \int_{s+K_T} f(x) {\,}
dx,
\end{equation}
where $K_T = \{ x \in \rr d:\ 0 \leq x_j \leq T, \
j=1,\dots,d \}$, exists independent of $s \in \rr d$. By $\mathscr M_x$ we
understand the mean value in the variable $x$ of a function of
several variables. The Bohr--Fourier transformation
(cf. \cite{Levitan1}) is defined by
\begin{equation}\nonumber
\mathscr F_{B} f (\la) = \wh f_\la = \mathscr M_x(f(x) e^{-2 \pi i \la \cdot x}), \quad \la \in \rr d,
\end{equation}
and $\wh f_\la \neq 0$ for at most countably many $\la \in \rr d$.
The set $\La = \La(f) = \{ \la \in \rr d: \wh f_\la \neq 0 \}$ is called the set of frequencies for $f$.

A function $f \in C_{\rm ap}(\rr d)$ may be reconstructed from its
Bohr--Fourier coefficients $(\wh f_\la)_{\la \in \La}$ using
Bochner--Fej{\' e}r polynomials \cite{Levitan1,Shubin4}. The reconstruction formula may be written as the uniform limit
\begin{equation}\label{fourierreconstruction1}
f(x) = \lim_{n \rightarrow \infty} \sum_{\la \in \La} K_n(\la) \ \wh f_\la \ e^{2 \pi i x \cdot \la}
\end{equation}
where $(K_n)_{n=1}^\infty$ is a sequence of functions on $\rr d$, such that for each $n \geq 1$ we have: The support of $K_n$ is finite, $K_n(\la)=0$ for $\la \notin \La$, and $0 \leq K_n(\la) \leq 1$ for all $\la \in \rr d$. Furthermore, $K_n (\la) \rightarrow 1$ as $n \rightarrow \infty$ for each $\la \in \La$.
For a set $\mathcal F \subseteq C_{\rm ap}(\rr d)$ that is precompact (or, synonymously, totally bounded) the limit \eqref{fourierreconstruction1} is uniform over $\mathcal F$ (see e.g. \cite[Lemma 1]{Wahlberg1}):
\begin{equation}\label{fourierreconstruction2}
\sup_{f \in \mathcal F} \sup_{x \in \rr d} \left| f(x) - \sum_{\la \in \La} K_n(\la) \wh f_\la \ e^{2 \pi i x \cdot \la} \right| \rightarrow 0, \quad n \rightarrow \infty.
\end{equation}
For $m \in \no$, the space $C_{\rm ap}^m(\rr d)$ is defined as all $f \in
C^m(\rr d)$ such that $\partial^\alpha f \in C_{\rm ap}(\rr d)$ for
$|\alpha| \leq m$, and $C_{\rm ap}^\infty(\rr d) = \bigcap_{m \in \no}
C_{\rm ap}^m(\rr d)$. Then $C_{\rm ap}^\infty = C_{\rm ap} \cap C_b^\infty$
\cite{Shubin4}.

The mean value defines an inner product
\begin{equation}\label{besicovitch1}
(f,g)_B = \mathscr M(f \overline g), \quad f,g \in C_{\rm ap}(\rr d),
\end{equation}
and $\| f \|_B=(f,f)_B^{1/2}$ defines a norm on $C_{\rm ap}(\rr d)$, since $\| f \|_B=0 \Rightarrow f=0$, due to Plancherel's formula for $C_{\rm ap}(\rr d)$ \cite{Levitan1}
\begin{equation}\label{plancherel1}
\mathscr M (|f|^2) = \sum_{\la \in \rr d} | \wh f_\la|^2.
\end{equation}
The completion of $C_{\rm ap}(\rr d)$ in the norm $\| \cdot \|_B$ is the
Hilbert space of Besicovitch a.p. functions $B^2(\rr d)$
\cite{Shubin1}. We have the isometric isomorphism $B^2(\rr d) \simeq L^2(\rrb d)$ where $L^2(\rrb d)$ denotes the square integrable functions on the Bohr compactification $\rrb d$, equipped with its Haar measure $\mu$, normalized to $\mu(\rrb d)=1$ \cite{Shubin4}.

For a Banach space $X$, the space of $X$-valued a.p. functions is denoted $C_{\rm ap}(\rr d,X)$.
It enjoys many of the properties of $C_{\rm ap}(\rr d)$, for example the equivalence among properties (i), (ii) and (iii) above (cf. \cite{Amerio1,Corduneanu1,Levitan1}). When $X$ is a Hilbert space we also have the natural generalization of Plancherel's formula \eqref{plancherel1}, and $C_{\rm ap}(\rr d,X)$ can be completed in the norm
\begin{equation}\nonumber
\| f \|_{B(\rr d, X)} = \left( \mathscr M_x ( \| f(x) \|_X^2 ) \right)^{1/2}
\end{equation}
to the Hilbert space-valued Besicovitch space $B^2(\rr d, X)$.

Analogous to the usual Sobolev space norm
\begin{equation}\nonumber
\| f \|_{H^s(\rr d)} = \left(\int_{\rr d} {\eabs \xi}^{2s} |\widehat
f(\xi)|^2 {\,} d\xi \right)^{1/2},
\end{equation}
Shubin \cite{Shubin1} has defined Sobolev--Besicovitch spaces of
a.p. functions $H^s(\rrb d)$ for $s \in \ro$, as the completion of
$TP(\rr d)$ in the norm corresponding to the inner product
\begin{equation}\nonumber
(f,g)_{H^s(\rrb d)} = \sum_{\xi \in \rr d} {\eabs \xi}^{2s} \wh f_\xi \
\overline{\wh g}_\xi, \quad f,g \in TP(\rr d).
\end{equation}
It follows that
\begin{equation}\label{bohrfourierunitary1}
\mathscr F_B: H^s(\rrb d) \mapsto l_s^2(\rr d) \quad \mbox{is unitary for any $s \in \ro$}.
\end{equation}
The spaces $H^s(\rrb d)$ are nonseparable Hilbert spaces, $H^0(\rrb d)=B^2(\rr d)$, and
$H^{\infty}(\rrb d) = \bigcap_{s \in \ro} H^s(\rrb d)$.
We have
$$
\| f \|_{H^s(\rrb d)} = \sup_{\| g \|_{H^{-s}(\rrb d)} \leq 1} |(f,g)_B|,
$$
and the dual $(H^s)'(\rrb d)$ can be identified isometrically with $H^{-s}(\rrb d)$ by means of $(\cdot,\cdot)_B$.
We have the embedding $C_{\rm ap}^\infty(\rr d) \subseteq H^{\infty}(\rrb
d)$, but there is no result corresponding to the Sobolev embedding
theorem for the Sobolev--Besicovitch spaces. In fact,
$H^{\infty}(\rrb d)$ is not embedded in $C_{\rm ap}(\rr d)$ \cite{Shubin1}.

We will use the family of H\" ormander symbol classes (cf. \cite{Folland1,Hormander3,Shubin5}),
with an almost periodicity condition in the space variables for each frequency (cf. \cite{Shubin1,Shubin2,Shubin3,Shubin4,Shubin4b}).
We impose the conditions
\begin{equation}\label{rhodeltakrav1}
0 < \rho \leq 1, \quad 0 \leq \delta < 1, \quad \delta
\leq \rho.
\end{equation}

\begin{defn}\label{aphormander1}
For $m \in \ro$ the space
$APS_{\rho,\delta}^m$ is defined as the space of all $a \in
C^\infty(\rr {2d})$ such that $a(\cdot,\xi) \in C_{\rm ap}(\rr d)$ for all $\xi \in \rr d$, and
\begin{equation}\label{hormclass1}
\sup_{x,\xi \in \rr d} \eabs \xi^{- m + \rho |\alpha| - \delta
|\beta|} |\partial_\xi^\alpha
\partial_x^\beta a(x,\xi)| < \infty, \quad \alpha,\beta \in
\nn d.
\end{equation}
\end{defn}

We consider the Kohn--Nirenberg quantization of pseudodifferential
operators, defined by
\begin{equation}\label{kndef2}
a(x,D) f(x) = \int_{\rr d} e^{2 \pi i \xi \cdot x} a(x,\xi) \wh f(\xi) {\,} d\xi, \quad f  \in \mathscr S (\rr d).
\end{equation}
If the symbol $a$ does not depend on $x$ then we write $a(D)$ instead of $a(x,D)$.
If $a \in APS_{\rho,\delta}^m$ then $a(x,D): \mathscr S (\rr d) \mapsto \mathscr S (\rr d)$ continuously.
In order to extend the operator to act on $C_b^\infty(\rr d)$ one modifies the definition \eqref{kndef2} into
\begin{equation}\label{kndef3}
a(x,D) f(x) = \lim_{\ep \rightarrow +0} \int_{\rr {2d}} \psi(\ep y)
\ \psi(\ep\xi) \ e^{2 \pi i \xi \cdot (x-y)} \ a(x,\xi) \ f(y) {\,} dy {\,}
d\xi
\end{equation}
where $\psi \in C_c^\infty(\rr d)$ equals one in a neighborhood of
the origin.
With this definition it can be shown (cf. \cite{Shubin1,Shubin4}) that $a(x,D):
C_{\rm ap}^\infty(\rr d) \mapsto C_{\rm ap}^\infty(\rr d)$ continuously if $a \in APS_{\rho,\delta}^m$.

The operators corresponding to the symbol classes $APS_{\rho,\delta}^m$ are called a.p. pseudodifferential operators, denoted $APL_{\rho,\delta}^m$, and following convention (cf. \cite{Shubin1,Shubin2,Shubin4}),  we set
\begin{equation}\nonumber
\begin{aligned}
APS_{\rho,\delta}^\infty & = \bigcup_{m \in \ro} APS_{\rho,\delta}^m, \quad APS^{-\infty} & = \bigcap_{m \in \ro} APS_{\rho,\delta}^m, \\
APL_{\rho,\delta}^\infty & = \bigcup_{m \in \ro} APL_{\rho,\delta}^m, \quad APL^{-\infty} & = \bigcap_{m \in \ro} APL_{\rho,\delta}^m,
\end{aligned}
\end{equation}
where the intersections do not depend on $\rho$, $\delta$.

Let $T_{0,-\xi} a(x,\eta) = a(x,\eta+\xi)$ denote translation in the second argument.
We note that
\begin{equation}\label{translationsymbol1}
(T_{0,-\xi} a)(x,D) = M_{-\xi} \circ a(x,D) \circ M_\xi
\end{equation}
holds for both definitions \eqref{kndef2} and \eqref{kndef3}.

For a pair of symbols $a$, $b$ such that $a(x,D) b(x,D)$ is well defined we define the
\emph{symbol product} $\wpr_0$ by
\begin{equation}\label{symbolproduct1}
c = a \wpr_0 b \quad \Longleftrightarrow \quad c(x,D) = a(x,D) \ b(x,D).
\end{equation}
For the symbol classes $APS_{\rho,\delta}^m$, the symbol product is a continuous bilinear map (cf. \cite{Folland1,Hormander3})
\begin{equation}\label{apsymbkont1}
APS_{\rho,\delta}^{m} \ \wpr_0 \ APS_{\rho,\delta}^{n} \subseteq
APS_{\rho,\delta}^{m+n}, \quad m,n \in \ro.
\end{equation}
In fact, it is shown in \cite[Thm.~3.1]{Shubin4} that $a \ \wpr_0 \ b \in APS_{\rho,\delta}^{m+n}$ for
$a \in APS_{\rho,\delta}^{m}$ and $b \in APS_{\rho,\delta}^{n}$, when $\rho=1$ and $\delta=0$. The proof extends to $\rho,\delta$ that satisfy \eqref{rhodeltakrav1}. For a proof of the continuity of the bilinear map \eqref{apsymbkont1} under more general assumptions we refer to \cite{Folland1,Hormander3}.

The following definition gives a sufficient condition for the operator $a(x,D)$ to be hypoelliptic (cf. \cite{Shubin1,Shubin5}).

\begin{defn}\label{aphypo1}
A symbol $a \in APS_{\rho,\delta}^m$ is called formally hypoelliptic \cite{Shubin5}, denoted $a \in APHS_{\rho,\delta}^{m,m_0}$, provided there exists $C,R > 0$ and $m_0 \leq m$ such that
\begin{equation}\label{hypoelliptic1}
\begin{aligned}
|a(x,\xi)| & \geq C {\eabs \xi}^{m_0}, \quad |\xi| \geq R \\
\left| \left( \pddm \xi \alpha \pddm x \beta a (x,\xi) \right) a(x,\xi)^{-1} \right| & \leq C_{\alpha,\beta} {\eabs \xi}^{-\rho|\alpha| + \delta|\beta|}, \quad |\xi| \geq R. \quad C_{\alpha,\beta}>0.
\end{aligned}
\end{equation}
\end{defn}

The space of operators $a(x,D)$ such that $a \in APHS_{\rho,\delta}^{m,m_0}$ is denoted $APHL_{\rho,\delta}^{m,m_0}$.
If $0 \leq \delta < \rho \leq 1$ and $a \in APHS_{\rho,\delta}^{m,m_0}$, there exists according to \cite[Thm.~5.1]{Shubin1} a symbol $b \in APHS_{\rho,\delta}^{-m_0,-m}$ such that
\begin{equation}\label{regularizer1}
b(x,D) a(x,D) = I - r(x,D), \quad a(x,D) b(x,D) = I - \wt r(x,D),
\end{equation}
where $r, \wt r \in APS^{-\infty}$. The operator $b(x,D)$ is called a parametrix.

Shubin has shown
\begin{equation}\label{normlikhet1}
a \in APS_{\rho,\delta}^0 \quad \Longrightarrow \quad \| a(x,D) \|_{\Lop(L^2)} = \| a(x,D) \|_{\Lop(B^2)} < \infty.
\end{equation}
In fact, the norm equality is proved in \cite[Thm.~5.1]{Shubin3} (see also \cite{Shubin4}) for $0 \leq \delta < \rho \leq 1$, and the proof extends to the assumption \eqref{rhodeltakrav1}.
If $a \in APS_{\rho,\delta}^m$ where $m>0$ then $a(x,D)$ is in general not bounded, neither on $L^2(\rr d)$ nor on $B^2(\rr d)$. Instead it may be considered an unbounded operator on either $L^2(\rr d)$ or $B^2(\rr d)$. From \eqref{apsymbkont1} it follows that $APL_{\rho,\delta}^{\infty}$ is an algebra of unbounded operators on $L^2(\rr d)$ and on $B^2(\rr d)$.

The space of linear unbounded closable operators on a Hilbert space $H$ is denoted $L(H)$. For $T \in L(H)$ we will denote by $\Ker \ T$, $\Ran \ T$, $\Dom \ T$, $\Coker \ T=H/\Ran \ T$, and $\overline T$,
its kernel, range, domain, cokernel, and closure, respectively. $\Dom \ T$ will always be dense in $H$.
The notation $\overline f$ is also used for the complex conjugate of a function $f$, and the closure $\overline A$ of a subset $A$ of a topological space,
ambiguity being avoided from the context.

An operator $A \in L(H)$ is positive on a
vector space $X \subseteq H$ if $(Af,f)_H \geq 0$ for all $f
\in X$.
This is denoted $A \geq 0$ where the spaces $X$ and $H$ are understood
from the context. We will use the following pairs $(X,H)$:
$(\mathscr S(\rr d),L^2(\rr d))$, $(TP(\rr d),B^2(\rr d))$,
$(l_f^2,l^2)$, $(\mathscr S(\rr d,l_f^2),L^2(\rr d,l^2))$ and
$( TP(\rr d, \mathscr S(\rr d)), B^2(\rr d, L^2(\rr d)))$.

Let $\mathscr A$ be an algebra of unbounded operators on a Hilbert space $H$. A common domain $\Dom \ {\mathscr A} \subseteq H$, dense in $H$, is assumed to exists for all $A \in \mathscr A$. Each operator $A \in \mathscr A$ is assumed to have a formal adjoint $A^+ \in \mathscr A$ that satisfies $(Af,g)_H=(f,A^+ g)_H$ for all $f,g \in \Dom \ {\mathscr A}$.
A \textit{representation} of $\mathscr A$ on $L(H')$, where $H'$ is a Hilbert space, is a linear map
$$
T: \mathscr A \mapsto L(H')
$$
that preserves operator composition and identity.
It is assumed that there exists a common domain $\Dom \ T {\mathscr A} \subseteq H'$, dense in $H'$.
The representation is called a faithful $^+$-representation if it preserves the formal adjoint operation and is injective.

\section{The representation by Gladyshev}

Gladyshev \cite{Gladyshev1} introduced a transformation of covariance functions of a certain type of second-order stochastic processes, called almost periodically correlated (or cyclostationary), which has been important in the development of the corresponding branch of probability theory. The transformation maps the covariance function (or operator kernel) of an almost periodically correlated stochastic process into an operator kernel of a translation-invariant operator, which may be interpreted as the covariance function of a vector-valued weakly stationary stochastic process. This fact is due to Gladyshev's result that the map preserves positivity.
Since the covariance operator of an almost periodically correlated stochastic process has a symbol that is almost periodic in the first variable, there is a connection to the theory of a.p. pseudodifferential operators.

Let us recall some definitions from \cite{Wahlberg1}
where Gladyshev's transformation is studied in the context of a.p. pseudodifferential operators.
Throughout this section we assume that \eqref{rhodeltakrav1} holds.
We denote the Bohr--Fourier coefficients of a symbol $a \in APS_{\rho,\delta}^m$ by
\begin{equation}\label{bohrfourier1}
a_\lambda (\xi) = \wh{a(\cdot,\xi)}_\la = \mathscr M_x( a(x,\xi) e^{- 2 \pi i
\lambda \cdot x}), \quad \xi \in \rr d, \quad \lambda \in \rr d.
\end{equation}
Then
\begin{equation}\nonumber
\Lambda = \Lambda(a) = \{ \lambda \in \rr d: \ \exists \xi \in \rr
d: \ a_\lambda(\xi) \neq 0 \}
\end{equation}
is countable \cite{Wahlberg1}. Without loss of generality we may assume that $\La$ is a linear space over $\qo$ with closure $\overline \La = \rr d$.
Based on \eqref{bohrfourier1} we define
\begin{equation}\label{Udef}
U(a)(\xi)_{\lambda,\lambda'} = a_{\lambda'-\lambda} (\xi-\lambda'),
\quad \lambda,\lambda', \xi \in \rr d, \quad a \in APS_{\rho,\delta}^m,
\end{equation}
which is a slight modification of the definition in \cite{Wahlberg1}. In fact, there we defined $U(a)(\xi)_{\lambda,\lambda'}$ only for $(\la,\la') \in \La \times \La$.
We consider $U(a)(\xi)$ as the kernel of an operator acting on
complex-valued sequences $(z_\la)_{\la \in S}$ with nonzero entries in a countable set $S \subseteq \rr d$.
The operator defined by $U(a)(\xi)$ acting on $z$, evaluated at index $\la \in \rr d$, is thus
$$
(U(a)(\xi) \cdot z)_{\la} = \sum_{\la' \in S} U(a)(\xi)_{\lambda,\lambda'} z_{\la'} = \sum_{\la' \in S} a_{\lambda'-\lambda} (\xi-\lambda') z_{\la'}
$$
which is zero unless $\la \in \La + S$, which is a countable set. Thus $U(a)(\xi)$ maps a sequence with countably many nonzero entries into another such sequence.
Since $a(x,D) e^{2 \pi i x \cdot \la} = a(x,\la) e^{2 \pi i x \cdot \la}$ (cf. \cite{Shubin5}) we have
\begin{equation}\label{innerproduct1}
\begin{aligned}
(a(x,D)f,g)_B & = \sum_{\la,\la' \in \rr d} \mathscr M_x( a(x,\lambda) e^{2 \pi i x
\cdot (\la-\la')} ) \ \wh f_\la \ \overline{\wh g}_{\la'} \\
& = \sum_{\la,\la' \in \rr d} a_{\la'-\la} (\la) \ \wh f_\la
\ \overline{\wh g}_{\la'} \\
& = ( U(a)(0) \cdot \mathscr F_B R f, \mathscr F_B R g )_{l^2}, \quad f,g \in TP(\rr d).
\end{aligned}
\end{equation}
By means of \eqref{bohrfourierunitary1} we thus obtain for $a \in APS_{\rho,\delta}^m$
\begin{equation}\label{matrixnorm1}
\begin{aligned}
\| U(a)(0) \|_{\Lop(l_s^2,l_{s-m}^2)}
& = \sup_{\| x \|_{l_s^2} \leq 1, \ \| y \|_{l_{m-s}^2} \leq 1} (U(a)(0) \ x ,y)_{l^2} \\
& = \sup_{\| f \|_{H^s(\rrb d)} \leq 1, \ \| g \|_{H^{m-s}(\rrb d)} \leq 1} ( a(x,D) f ,g)_{B} \\
& = \| a(x,D) \|_{\Lop(H^s(\rrb d), H^{s-m}(\rrb d))} \\
& = \| a(x,D) \|_{\Lop(H^s(\rr d), H^{s-m}(\rr d))} < \infty, \quad s \in \ro,
\end{aligned}
\end{equation}
where the last equality is a consequence of \eqref{apsymbkont1} and \eqref{normlikhet1} (see e.g. \cite[Cor.~1]{Wahlberg1}).
Moreover, \eqref{innerproduct1} gives
\begin{equation}\label{matrixpositivity1}
a(x,D) \geq 0 \quad \mbox{on} \quad TP(\rr d) \quad \Longleftrightarrow \quad U(a)(0) \geq 0 \quad  \mbox{on} \quad l_f^2.
\end{equation}

The operator-valued function $\xi \mapsto U(a)(\xi)$ may be used to define a Fourier multiplier operator
\begin{equation}\label{kvantvekt0}
U(a)(D) F(x) = \int_{\rr d} e^{2 \pi i \xi \cdot x} \ U(a)(\xi) \cdot \wh F(\xi) {\,} d\xi
\end{equation}
acting on vector-valued functions
\begin{equation}\nonumber
\rr d \ni x \mapsto ( F_\la (x) )_{\la \in \La},
\end{equation}
where initially we let $F(x) = (F_\la(x))_{\la \in \La} \in \mathscr S(\rr d,l_f^2)$.
We denote the map $a(x,D) \mapsto U(a)(D)$ by
\begin{equation}\label{gladyshevrepr1}
\wt U(a(x,D)) = U(a)(D), \quad a \in APS_{\rho,\delta}^m.
\end{equation}
Since
\begin{equation}\label{onequant1}
U(1)(\xi)_{\la,\la'} = \delta_{(\la'-\la)} = I_{l^2}, \quad \xi \in \rr d,
\end{equation}
we have $\wt U(I)(D)=I$ where $I$ denotes the identity operator, both on $\mathscr S(\rr d)$ and on $\mathscr S(\rr d,l_f^2)$.

Next we study continuity and growth properties of the operator-valued function $\xi \mapsto U(a)(\xi)$.
The following result improves \cite[Prop.~3]{Wahlberg1}.

\begin{prop}\label{operatorkont1}
If $a \in APS_{\rho,\delta}^m$ then we have for any $s \in \ro$
\begin{align}
& U(a) \in C^\infty (\rr d, \Lop(l_s^2,l_{s-m}^2)),  \label{opkont1a} \\
& C_s^{-1} {\eabs \xi}^{-|s|-|m-s|} \leq \| U(a)(\xi) \|_{\Lop(l_{s}^2,l_{s-m}^2)} \leq C_s {\eabs
\xi}^{|s|+|m-s|}, \quad C_s >0. \label{opkont1b}
\end{align}
If $m \leq 0$ then we have the isometry
\begin{equation}\label{operatorisom1}
\| U(a)(\xi) \|_{\Lop(l^2)} = \| U(a)(0) \|_{\Lop(l^2)}, \quad \xi \in \rr d.
\end{equation}
\end{prop}
\begin{proof}
By \cite[Prop.~3]{Wahlberg1} we have $U(a) \in C (\rr d, \Lop(l_s^2,l_{s-m}^2))$.
Since $\pddm \xi \alpha (U(a))(\xi) = U(\pddm \xi \alpha a)(\xi)$ (cf. \cite{Wahlberg1}) and $\pddm \xi \alpha a \in APS_{\rho,\delta}^{m-\rho|\alpha|} \subseteq APS_{\rho,\delta}^{m}$, the result \eqref{opkont1a} follows.

In order to prove \eqref{opkont1b}, let $\xi \in \rr d$.
Since $U(a)(\xi)=U(T_{0,-\xi} a)(0)$ we get using \eqref{translationsymbol1} and \eqref{matrixnorm1}
\begin{equation}\nonumber
\begin{aligned}
\| U(a)(\xi)\|_{\Lop(l_{s}^2,l_{s-m}^2)} & = \| ( T_{0,-\xi} a)(x,D) \|_{\Lop(H^s(\rr d),H^{s-m}(\rr d))} \\
& = \sup_{\| f \|_{H^s} \leq 1, \ \| g \|_{H^{m-s}} \leq 1} \left| (M_{-\xi} \ a(x,D) \ M_\xi f,g)_{L^2} \right| \\
& = \sup_{\| M_{-\xi} f \|_{H^s} \leq 1, \ \| M_{-\xi} g \|_{H^{m-s}} \leq 1} \left| (a(x,D) f,g)_{L^2} \right|.
\end{aligned}
\end{equation}
Since $\| f \|_{H^s(\rr d)} \leq C_s {\eabs \xi}^{|s|} \| M_{-\xi} f \|_{H^s(\rr d)}$ for some $C_s>0$ we obtain
\begin{equation}\nonumber
\begin{aligned}
& \| U(a)(\xi)\|_{\Lop(l_{s}^2,l_{s-m}^2)} = C_s {\eabs \xi}^{|s|} C_{m-s} {\eabs \xi}^{|m-s|} \\
& \times \sup_{\| M_{-\xi} f \|_{H^s} \leq 1, \ \| M_{-\xi} g \|_{H^{m-s}} \leq 1} \left| ( a(x,D) C_s^{-1} {\eabs \xi}^{-|s|} f, C_{m-s}^{-1} {\eabs \xi}^{-|m-s|} g)_{L^2} \right| \\
& \leq C {\eabs \xi}^{|s|+ |m-s|} \| a(x,D) \|_{\Lop(H^{s},H^{s-m})} = C {\eabs \xi}^{|s|+ |m-s|} \| U(a)(0) \|_{\Lop(l_{s}^2,l_{s-m}^2)},
\end{aligned}
\end{equation}
which proves the upper bound \eqref{opkont1b}. The lower bound follows from
\begin{equation}\nonumber
\begin{aligned}
\| U(a)(0)\|_{\Lop(l_{s}^2,l_{s-m}^2)} & = \| U(T_{0,-\xi} a)(-\xi)\|_{\Lop(l_{s}^2,l_{s-m}^2)} \\
& \leq C {\eabs \xi}^{|s|+ |m-s|} \| U(T_{0,-\xi} a)(0) \|_{\Lop(l_{s}^2,l_{s-m}^2)} \\
& = C {\eabs \xi}^{|s|+ |m-s|} \| U(a)(\xi) \|_{\Lop(l_{s}^2,l_{s-m}^2)}.
\end{aligned}
\end{equation}
Finally, if $m \leq 0$ then we have for $\xi \in \rr d$
\begin{equation}\nonumber
\begin{aligned}
\| U(a)(\xi)\|_{\Lop(l^2)} & = \| ( T_{0,-\xi} a)(x,D) \|_{\Lop(L^2(\rr d))} \\
& = \sup_{\| f \|_{L^2} \leq 1, \ \| g \|_{L^2} \leq 1} \left| (M_{-\xi} \ a(x,D) \ M_\xi f,g)_{L^2} \right| \\
& = \sup_{\| M_{-\xi} f \|_{L^2} \leq 1, \ \| M_{-\xi} g \|_{L^2} \leq 1} \left| (a(x,D) f,g)_{L^2} \right| \\
& = \| a(x,D) \|_{\Lop(l^2)} = \| U(a)(0)\|_{\Lop(l^2)},
\end{aligned}
\end{equation}
proving \eqref{operatorisom1}.
\end{proof}

\begin{cor}\label{sobolevcorollary1}
If $a \in APS_{\rho,\delta}^m$ then for any $s\in \ro$
\begin{equation}\label{sobolevkont1}
U(a)(D): \ H^{s}(\rr d, l_{s}^2) \ \mapsto \ H^{s-|s|-|m-s|}(\rr d, l_{s-m}^2)
\end{equation}
continuously.
If $m \leq 0$ then $U(a)(D) \in \Lop(L^2(\rr d, l^2))$.
\end{cor}
\begin{proof}
Let $F \in \mathscr S(\rr d, l_{f}^2)$. We obtain using \eqref{opkont1b}
\begin{equation}\nonumber
\begin{aligned}
& \| U(a)(D)F \|_{H^{s-|s|-|m-s|}(\rr d, l_{s-m}^2)}^2 \\
& = \int_{\rr d} \| U(a)(\xi) \cdot \wh F(\xi) \|_{l_{s-m}^2}^2 {\eabs \xi}^{2(s-|s|-|m-s|)} d\xi \\
& \leq \int_{\rr d} \| U(a)(\xi) \|_{\Lop(l_{s}^2,l_{s-m}^2)}^2 \| \wh F(\xi) \|_{l_{s}^2}^2 {\eabs \xi}^{2(s-|s|-|m-s|)} d\xi \\
& \leq C \| F \|_{H^{s}(\rr d, l_{s}^2)}^2.
\end{aligned}
\end{equation}
The density of $\mathscr S(\rr d, l_{f}^2)$ in $H^{s}(\rr d, l_{s}^2)$ now proves \eqref{sobolevkont1}. Finally, if $m \leq 0$ then $U(a)(D) \in \Lop(L^2(\rr d, l^2))$ follows from \eqref{operatorisom1} and Plancherel's theorem for $L^2(\rr d,l^2)$.
\end{proof}

Next we prove a result that simplifies the proofs of \cite[Prop.~4 and Thm.~2]{Wahlberg1}.

\begin{lem}\label{simplification1}
If $a \in APS_{\rho,\delta}^m$ then for any $\xi \in \rr d$ and any $s \in \ro$
\begin{equation}\label{Uekviv1}
U(a)(\xi) = (\mathscr F_{B} R M_{-\xi}) \ a(x,D) \ (\mathscr F_{B} R M_{-\xi})^*: \ l_s^2 \mapsto l_{s-m}^2
\end{equation}
is continuous.
\end{lem}
\begin{proof}
From Proposition \ref{operatorkont1} we know that $U(a)(\xi): l_s^2 \mapsto l_{s-m}^2$ is continuous for any $s \in \ro$ and any $\xi \in \rr d$,
and \eqref{innerproduct1} gives the factorization
\begin{equation}\label{unitaryequiv1}
U(a)(0) = \mathscr F_{B} R \ a(x,D) \ (\mathscr F_{B} R)^*.
\end{equation}
Combined with \eqref{translationsymbol1} this gives
\begin{equation}\nonumber
\begin{aligned}
U(a)(\xi) & = U(T_{0,-\xi} a)(0) = \mathscr F_{B} R \ (T_{0,-\xi} a)(x,D) \ (\mathscr F_{B} R)^* \\
& = (\mathscr F_{B} R M_{-\xi}) \ a(x,D) \ (\mathscr F_{B} R M_{-\xi})^*.
\end{aligned}
\end{equation}
\end{proof}

In order to formulate the following corollary of \eqref{unitaryequiv1}, we need a result by Shubin (cf. \cite[Thm.~3.4]{Shubin1}, \cite[Thm.~3.2 and Cor.~4.1]{Shubin4} and \cite[Thm.~18.1.7 and p.~94]{Hormander3}).
If $APL_{\rho,\delta}^m$ then there exists a formal adjoint operator $a(x,D)^+ = a^+(x,D) \in APL_{\rho,\delta}^m$ with symbol $a^+ \in APS_{\rho,\delta}^m$, that satisfies
\begin{equation}\nonumber
\begin{aligned}
(a(x,D)f,g)_{L^2} & = (f,a(x,D)^+ g)_{L^2}, \quad f,g \in \mathscr S(\rr d), \\
(a(x,D)f,g)_{B} & = (f,a(x,D)^+ g)_{B}, \quad f,g \in TP(\rr d).
\end{aligned}
\end{equation}
We denote by $U(a)(\xi)^+_{\la,\la'} = \overline{U(a)(\xi)}_{\la',\la}$ the Hermitean conjugation of the kernel $U(a)(\xi)$.
We use $l_f^2$ as a common domain for all operators $\{ U(a)(0), \ a \in APS_{\rho,\delta}^\infty \}$.

\begin{cor}\label{unitaryequivalent1}
The operators
\begin{equation}\nonumber
U(a)(0): l_s^2 \mapsto l_{s-m}^2 \quad \mbox{and} \quad a(y,D): H^s(\rrb d) \mapsto H^{s-m}(\rrb d)
\end{equation}
are unitarily equivalent for any $s \in \ro$. Considering the algebra $APL_{\rho,\delta}^\infty$ either as a subspace in $L(B^2(\rr d))$ or in $L(L^2(\rr d))$, the map
$$
a(y,D) \mapsto U(a)(0)
$$
is a faithful $^+$-representation of $APL_{\rho,\delta}^\infty$ on $L(l^2(\rr d))$. It is positivity preserving in the sense that $a(x,D) \geq 0$ on $TP(\rr d)$ if and only if $U(a)(0) \geq 0$ on $l_{f}^2$.
\end{cor}

We note that $a(x,D) \geq 0$ on $TP(\rr d)$ is equivalent to $a(x,D) \geq 0$ on $\mathscr S(\rr d)$ (cf. \cite{Shubin4} and \cite[Cor.~2]{Wahlberg1}).
As a consequence of Lemma \ref{simplification1} we also obtain the following result which contains \cite[Prop.~4 and Thm.~2]{Wahlberg1}.

\begin{prop}\label{Ucomp1}
If $a \in APS_{\rho,\delta}^m$ and $b \in APS_{\rho,\delta}^n$ then
\begin{equation}\label{Urepr1}
U(a \wpr_0 b)(\xi) = U(a)(\xi) \cdot U(b)(\xi), \quad \xi \in \rr d.
\end{equation}
We have $a(x,D) \geq 0$ on $TP(\rr d)$ if and only if $U(a)(\xi) \geq 0$ on $l_{f}^2$ for all $\xi \in \rr d$.
\end{prop}
\begin{proof}
We obtain from \eqref{apsymbkont1}, Lemma \ref{simplification1} and \eqref{symbolproduct1}
\begin{equation}\nonumber
\begin{aligned}
& U(a \wpr_0 b)(\xi) \\
& = (\mathscr F_{B} R M_{-\xi}) \ a(x,D) \ b(x,D) \ (\mathscr F_{B} R M_{-\xi})^* \\
& = (\mathscr F_{B} R M_{-\xi}) \ a(x,D) \ (\mathscr F_{B} R M_{-\xi})^* \ (\mathscr F_{B} R M_{-\xi}) \ b(x,D) \ (\mathscr F_{B} R M_{-\xi})^* \\
& = U(a)(\xi) \cdot U(b)(\xi), \quad \xi \in \rr d,
\end{aligned}
\end{equation}
which proves \eqref{Urepr1}.
From \eqref{translationsymbol1} we may conclude that $a(x,D) \geq 0$ on $TP(\rr d)$ is equivalent to $(T_{0,-\xi} a)(x,D) \geq 0$ on $TP(\rr d)$ for any $\xi \in \rr d$, which by \eqref{matrixpositivity1} is equivalent to $U(a)(\xi) \geq 0$ on $l_f^2$ for any $\xi \in \rr d$.
\end{proof}

We get consequently the following result. Here we use $\mathscr S(\rr d, l_f^2)$ as a common domain of all operators $\{ U(a)(D), \ a \in APS_{\rho,\delta}^\infty \}$.

\begin{cor}\label{representation1}
Considering the algebra $APL_{\rho,\delta}^\infty$ either as a subspace in $L(B^2(\rr d))$ or in $L(L^2(\rr d))$, the map
\begin{equation}\nonumber
a(x,D) \mapsto \wt U (a(x,D)) = U(a)(D)
\end{equation}
is a faithful $^+$-representation of $APL_{\rho,\delta}^\infty$ on $L( L^2(\rr d, l^2) )$.
It preserves positivity in the sense that $a(x,D) \geq 0$ on $TP(\rr d)$ if and only if $U(a)(D) \geq 0$ on $\mathscr S(\rr d,l_{f}^2)$.
\end{cor}
\begin{proof}
For $a \in APS_{\rho,\delta}^m$ and $b \in APS_{\rho,\delta}^n$ we have by \eqref{symbolproduct1} and \eqref{Urepr1}
\begin{equation}\nonumber
\begin{aligned}
\wt U (a(x,D) b(x,D)) & = \wt U(a \wpr_0 b(x,D)) = U(a)(D) \ U(b)(D) \\
& = \wt U(a(x,D) \ \wt U(b(x,D)).
\end{aligned}
\end{equation}
By Lemma \ref{simplification1},
\begin{equation}\nonumber
U(a^+)(\xi) = (\mathscr F_{B} R M_{-\xi}) \ a(x,D)^+ \ (\mathscr F_{B} R M_{-\xi})^* = U(a)(\xi)^+.
\end{equation}
This gives for $F,G \in \mathscr S(\rr d,l_f^2)$
\begin{equation}\nonumber
\begin{aligned}
\left( F, U (a)(D)^+ G \right)_{L^2(\rr d,l^2)}
& = \left( U (a)(D) F, G \right)_{L^2(\rr d,l^2)} \\
& = \iint_{\rr {2d}} e^{2 \pi i x \cdot \xi} (U(a)(\xi) \cdot \wh F(\xi),G(x))_{l^2} \ d \xi dx \\
& = \int_{\rr d} (U(a)(\xi) \cdot \wh F(\xi), \wh G(\xi))_{l^2} \ d \xi \\
& = \iint_{\rr {2d}} ( F(x), e^{2 \pi i x \cdot \xi} U(a)(\xi)^+ \cdot \wh G(\xi))_{l^2} \ d \xi dx \\
& = \left( F, U (a^+)(D) \ G \right)_{L^2(\rr d,l^2)},
\end{aligned}
\end{equation}
that is, $\wt U (a(x,D))^+ = \wt U (a(x,D)^+)$ which proves that $\wt U$ is a $^+$-repres-entation.
If $U(a)(D)=0$ then $U(a)(\xi)_{\la,\la'}=0$ for all $\xi, \la, \la' \in \rr d$, which implies that $a=0$ due to the fact that the Bohr--Fourier inversion formula \eqref{fourierreconstruction1} gives
\begin{equation}\nonumber
a(x,\xi) = \lim_{n \rightarrow \infty} \sum_{\la \in \La} K_n(\la)
\ U(a)(\xi)_{-\la,0} \ e^{2 \pi i \la \cdot x}.
\end{equation}
Thus $a(x,D)=0$ and the representation is faithful.

If $a(x,D) \geq 0$ on $TP(\rr d)$ then $U(a)(\xi) \geq 0$ on $l_{f}^2$ for all $\xi \in \rr d$ by Proposition \ref{Ucomp1}. If $F \in \mathscr S(\rr d,l_{f}^2)$ we thus have
\begin{equation}\nonumber
\left( U (a)(D) F,F \right)_{L^2(\rr d,l^2)} = \int_{\rr d} (U(a)(\xi) \cdot \wh F(\xi), \wh F(\xi) )_{l^2} \ d \xi \geq 0
\end{equation}
which proves that $\wt U$ preserves positivity. Suppose on the other hand that $U (a)(D) \geq 0$ on $\mathscr S(\rr d,l_{f}^2)$. Let $z \in l_f^2$ and pick $\varphi \in
C_c^\infty(\rr d)$ with support in the unit ball such that $\varphi
\geq 0$ and $\| \varphi \|_{L^2}=1$. With
$\varphi_\ep(x)=\ep^{-d/2}\varphi(x/\ep)$ and $F_\ep(x)_\la=\mathscr
F^{-1} \varphi_\ep(x) z_\la$ we then have
\begin{equation}\nonumber
\begin{aligned}
0 \leq \left( U(a)(D) F_\ep, F_\ep \right)_{L^2(\rr d,l^2)} & = \int_{\rr d}
( U(a)(\xi) \cdot z, z)_{l^2} \ \varphi_\ep(\xi)^2 \ d\xi  \\
& \longrightarrow (U(a)(0) \cdot z, z)_{l^2}, \quad \ep
\longrightarrow 0,
\end{aligned}
\end{equation}
where we have used Proposition \ref{operatorkont1} and the shrinking support of
$\varphi_\ep$. Therefore $U(a)(0) \geq 0$ on $l_f^2$ which
implies that $a(x,D) \geq 0$ on $TP(\rr d)$ according to
\eqref{matrixpositivity1}.
\end{proof}

\section{The representation by Coburn, Moyer and Singer}

In this section we always assume that \eqref{rhodeltakrav1} holds.
For $x \in \rr d$ and $a \in APS_{\rho,\delta}^m$ the symbol $a_x$ is defined by
\begin{equation}\nonumber
a_x(y,\xi)=a(x+y,\xi) = (T_{-x,0} a)(y,\xi) \in APS_{\rho,\delta}^m.
\end{equation}
It follows from \eqref{kndef2} and \eqref{kndef3} that
$$
a_x(y,D) = T_{-x} \circ a(y,D) \circ T_x.
$$
Abbreviating $H^s=H^s(\rr d)$, we have
\begin{equation}\label{sobolevtransl1}
\| a_x(y,D) \|_{ \Lop (H^s,H^{s-m}) } = \| a(y,D) \|_{ \Lop (H^s,H^{s-m}) }, \quad x \in \rr d.
\end{equation}

In \cite{Coburn1}, a linear transformation
\begin{equation}\label{representation3}
a(y,D) \mapsto A(a(y,D)) := A
\end{equation}
is defined, such that $A$ is an operator acting on a function of two variables $u: \rr d \times \rr d \mapsto \co$, according to
\begin{equation}\label{Adescr1}
\begin{aligned}
(A u)(x,y) & = \int_{\rr d} e^{2 \pi i \xi \cdot y} a(x+y,\xi) \ \mathscr F_2 u(x,\xi) \ d \xi \\
& = \left( a_x(y,D) \ u(x,\cdot) \right) (y),
\end{aligned}
\end{equation}
where $\mathscr F_2$ denotes partial Fourier transformation in the second $\rr d$ variable.
The operator $A$ is well defined for $a \in APS_{\rho,\delta}^m$, for example if $u \in C_{\rm ap}(\rr d) \otimes \mathscr S(\rr d)$. The study of this transformation is developed further in \cite{Shubin2,Shubin4b}.

In order to prove a result about the continuity of the operator $A$ we need the following lemma.

\begin{lem}\label{hilbertiso1}
For any $s \in \ro$ we have the Hilbert space isomorphism
\begin{equation}\nonumber
B^2(\rr d) \otimes H^s(\rr d) \simeq B^2(\rr d, H^s(\rr d) ).
\end{equation}
\end{lem}

\begin{proof}
Let $\{\varphi_n \}_{n=0}^{\infty}$ be an ONB for $H^s(\rr d)$. Since $\{ e_\la \}_{\la \in \rr d}$ is an ONB for $B^2(\rr d)$, \cite[Prop.~II.4.2]{Reed1} implies that $\{ e_\la \otimes \varphi_n \}_{\la \in \rr d, n \in \no}$ is an ONB for $B^2(\rr d) \otimes H^s(\rr d)$.
Define $T(e_\la \otimes \varphi_n) = e_\la \varphi_n$.
Since $\{ e_\la \varphi_n \}_{\la \in \rr d, n \in \no}$ is an orthonormal system in $B^2(\rr d, H^s(\rr d) )$, $T$ extends by linearity to a continuous map $T: B^2(\rr d) \otimes H^s(\rr d) \mapsto B^2(\rr d, H^s(\rr d) )$ that preserves inner products. It remains to prove that $T$ is onto.
Suppose $f \in B^2(\rr d, H^s(\rr d) )$ and
$$
(f, e_\la \varphi_n)_{ B^2(\rr d, H^s(\rr d)) } = \mathscr M_x \left( (f(x),\varphi_n)_{H^s(\rr d)} \overline{e_\la (x)} \right) = 0, \quad \la \in \rr d, \ n \in \no.
$$
Since $\{ e_\la \}_{\la \in \rr d}$ is an ONB for $B^2(\rr d)$, this means that $x \mapsto (f(x),\varphi_n)_{H^s(\rr d)} = 0$ in $B^2(\rr d)$ for all $n \in \no$ $\Leftrightarrow$ $\mathscr M_x( |(f(x),\varphi_n)_{H^s(\rr d)}|^2 ) = 0$ for all $n \in \no$. This is equivalent to
$$
0 = \mathscr M_x \left( \sum_{n=0}^\infty |(f(x),\varphi_n)_{H^s(\rr d)}|^2 \right) = \mathscr M_x \left( \| f(x) \|_{H^s(\rr d)}^2 \right),
$$
that is $f=0$ in $B^2(\rr d, H^s(\rr d) )$.
Thus $\{ e_\la \varphi_n \}_{\la \in \rr d, n \in \no}$ is an ONB in $B^2(\rr d, H^s(\rr d) )$ and $T$ is unitary.
\end{proof}

\begin{prop}\label{Acont1}
For $a \in APS_{\rho,\delta}^m$ the map $A$ extends to a continuous map
\begin{equation}\label{Acontstatement1}
A: B^2(\rr d) \otimes H^s(\rr d) \mapsto B^2(\rr d) \otimes H^{s-m}(\rr d), \quad s \in \ro.
\end{equation}
\end{prop}
\begin{proof}
We abbreviate $H^s=H^s(\rr d)$, $\mathscr S = \mathscr S(\rr d)$ and $B^2=B^2(\rr d)$.
Let $s \in \ro$ and $f \in \mathscr S$.

As a first step we claim that $x \mapsto a_x(y,D)f$ extends to a strongly measureable function $\rrb d \mapsto H^{s-m}$.

In fact, $\{ a(\cdot,\xi) \ \wh f(\xi) \eabs{\xi}^{d+1} \}_{\xi \in \rr d} \subseteq C_{\rm ap}(\rr d)$ is a precompact family of functions, since it depends continuously on $\xi$ in the $C_{\rm ap}(\rr d)$ norm \cite{Wahlberg1} and decays to zero at infinity. Thus, according to \eqref{fourierreconstruction2} there exists for any $\ep>0$ a positive integer $N_\ep$ such that
\begin{equation}\label{fourierreconstruction3}
\sup_{\xi \in \rr d} \sup_{z \in \rr d} \left| \left( a(z,\xi) - \sum_{\la \in \La} K_n(\la) \ a_\la (\xi) \ e^{2 \pi i z \cdot \la} \right) \ \wh f(\xi) \eabs{\xi}^{d+1} \right| < \ep, \quad n  \geq N_\ep.
\end{equation}
Since $|a_\la(\xi)| \leq C \eabs{\xi}^m$ we have for $g \in \mathscr S$
\begin{equation}\nonumber
p_n (x) := \sum_{\la \in \La} K_n(\la) \ e^{2 \pi i x \cdot \la} \iint_{\rr {2d}} \ e^{2 \pi i y \cdot (\xi + \la)} a_\la (\xi) \ \wh f(\xi) \ \overline{g(y)} \ d\xi dy \in TP(\rr d).
\end{equation}
Combining with \eqref{fourierreconstruction3} it follows that
\begin{equation}\nonumber
\begin{aligned}
& |(a_x(y,D) f, g)_{L^2} - p_n (x)| \\
& = \left| \iint_{\rr {2d}} e^{2 \pi i y \cdot \xi} \left( a(x+y,\xi) - \sum_{\la \in \La} K_n(\la) e^{2 \pi i (x+y) \cdot \la} a_\la (\xi) \right) \wh f(\xi) \ \overline{g(y)} \ d\xi \ dy \right| \\
& \leq \ep \ \| g \|_{L^1} \ \| \eabs{\cdot}^{-d-1} \|_{L^1}, \quad x \in \rr d, \quad n  \geq N_\ep.
\end{aligned}
\end{equation}
It follows that the function $x \mapsto (a_x(y,D) f, g)_{L^2} \in C_{\rm ap}(\rr d)$, because it is a uniform limit of trigonometric polynomials. Next let $g \in H^{m-s}$ and pick $(g_n) \subseteq \mathscr S$ such that $\| g-g_n \|_{H^{m-s}} \rightarrow 0$ as $n \rightarrow \infty$.
Then \eqref{sobolevtransl1} gives
\begin{equation}\nonumber
\begin{aligned}
& |(a_x(y,D) f, g)_{L^2} - (a_x(y,D) f, g_n)_{L^2}| \leq \| a_x(y,D) f \|_{H^{s-m}} \| g - g_n \|_{H^{m-s}} \\
& \leq \| a(y,D) \|_{\Lop(H^s,H^{s-m})} \| f \|_{H^{s}} \| g - g_n \|_{H^{m-s}}, \quad x \in \rr d,
\end{aligned}
\end{equation}
which implies that $x \mapsto (a_x(y,D) f, g)_{L^2} \in C_{\rm ap}(\rr d)$ for any $g \in H^{m-s}$. Since a function in $C_{\rm ap}(\rr d)$ can be extended to a function in $C(\rrb d)$ (cf. \cite{Shubin4}), we may conclude that $\rrb d \ni x \mapsto (a_x(y,D) f, g)_{L^2}$ is a measurable function for any $g \in H^{m-s}$.
Since the dual $(H^{s-m})'$ can be identified with $H^{m-s}$ via the form $(\cdot,\cdot)_{L^2}$, this means that $\rrb d \ni x \mapsto a_x(y,D) f$ is a weakly measureable function $\rrb d \mapsto H^{s-m}$.
By \cite[Thm.~IV.22]{Reed1} the function $\rrb d \ni x \mapsto a_x(y,D) f$ is strongly measureable. Thus we have proved our claim:
\begin{equation}\label{measurable1}
\rrb d \ni x \mapsto a_x(y,D) f  \in H^{s-m} \quad \mbox{is strongly measureable}.
\end{equation}
Let $x \mapsto u(x,\cdot) \in TP(\rr d, \mathscr S)$, which means that $u$ has the form
\begin{equation}\label{trigpolsobolev1}
u(x,\cdot) = \sum_{j=1}^n e^{2 \pi i \xi_j \cdot x} f_j, \quad \xi_j \in \rr d, \quad f_j \in \mathscr S.
\end{equation}
Then
\begin{equation}\nonumber
A u(x,\cdot) = a_x(y,D) \ u(x,\cdot) = \sum_{j=1}^n e^{2 \pi i \xi_j \cdot x} a_x(y,D) f_j
\end{equation}
extends by \eqref{measurable1} to a strongly measureable function $\rrb d \mapsto H^{s-m}$.
Therefore it may be integrated.
We obtain using Lemma \ref{hilbertiso1}, \eqref{sobolevtransl1}, and denoting the Haar measure on $\rrb d$ by $\mu$,
\begin{equation}\nonumber
\begin{aligned}
\| A u \|_{B^2 \otimes H^{s-m}}^2
& = \mathscr M_x \left( \| a_x(y,D) \ u(x,\cdot) \|_{H^{s-m}}^2 \right) \\
& = \int_{\rrb d} \| a_x(y,D) \ u(x,\cdot) \|_{H^{s-m}}^2 \ \mu(dx) \\
& \leq \int_{\rrb d} \| a_x(y,D)\|_{\Lop(H^s,H^{s-m})}^2 \| u(x,\cdot) \|_{H^{s}}^2 \ \mu(dx) \\
& = \| a(y,D)\|_{\Lop(H^s,H^{s-m})}^2 \int_{\rrb d} \| u(x,\cdot) \|_{H^{s}}^2 \ \mu(dx) \\
& = \| a(y,D)\|_{\Lop(H^s,H^{s-m})}^2 \| u \|_{B^2 \otimes H^{s}}^2.
\end{aligned}
\end{equation}
Finally \eqref{Acontstatement1} follows from the density of $TP(\rr d, \mathscr S)$ in $B^2(\rr d,H^s) \simeq B^2 \otimes H^s$.
\end{proof}

It follows in particular from Proposition \ref{Acont1} that if for $a \in APS_{\rho,\delta}^0$ then $A \in \Lop(B^2(\rr d) \otimes L^2(\rr d))$.

Next we prove that the map \eqref{representation3} is a representation.
We use $TP(\rr d,\mathscr S)$ as a common domain of all operators $\{ A(a(y,D)), \ a \in APS_{\rho,\delta}^\infty \}$.

\begin{prop}\label{representation2}
Considering the algebra $APL_{\rho,\delta}^\infty$ either as a subspace in $L(B^2(\rr d))$ or in $L(L^2(\rr d))$, the map
$$
a(y,D) \mapsto A(a(y,D))
$$
is a faithful $^+$-representation of $APL_{\rho,\delta}^\infty$ on $L( B^2(\rr d) \otimes L^2(\rr d) )$.
It preserves positivity in the sense that $a(y,D) \geq 0$ on $TP(\rr d)$ if and only if $A \geq 0$ on $TP(\rr d,\mathscr S(\rr d))$.
\end{prop}
\begin{proof}
Again we abbreviate $B^2=B^2(\rr d)$, $H^s = H^s(\rr d)$ and $\mathscr S=\mathscr S(\rr d)$.
First we prove that if $a \in APS_{\rho,\delta}^m$, $b \in APS_{\rho,\delta}^n$, $A=A(a(y,D))$ and $B=B(b(y,D))$ then $A( a \wpr_0 b (y,D) ) = A \circ B$.
We have (cf. \cite{Hormander3})
\begin{equation}\nonumber
a \wpr_0 b(x,\xi) = e^{2 \pi i D_z \cdot D_\eta} a(x,\eta) b(z,\xi) \Big|_{z=x, \ \eta = \xi}
\end{equation}
where we denote $e^{2 \pi i D_z \cdot D_\eta} = \mathscr F^{-1} \mathcal M \mathscr F$ and $\mathcal M$ is the multiplier operator $(\mathcal M f) (\wh z,\wh \eta) = e^{2 \pi i \wh z \cdot \wh \eta} f(\wh z,\wh \eta)$ in the Fourier domain.
Hence
\begin{equation}\nonumber
\begin{aligned}
a \wpr_0 b(x+y,\xi) & = \mathscr F^{-1} \left( e^{2 \pi i z \cdot \eta} \mathscr F \left( a(x+y,\cdot) b(\cdot,\xi) \right) \right) \Big|_{z=x+y, \ \eta = \xi} \\
& = \mathscr F^{-1} \left( e^{2 \pi i z \cdot (\eta+x) } \mathscr F \left( (T_{-x,0} a)(y,\cdot) b(\cdot,\xi) \right) \right) \Big|_{z=y, \ \eta = \xi} \\
& = \mathscr F^{-1} \left( e^{2 \pi i z \cdot \eta} \mathscr F \left( (T_{-x,0} a)(y,\cdot) (T_{-x,0} b)(\cdot,\xi) \right) \right) \Big|_{z=y, \ \eta = \xi} \\
& = (T_{-x,0} a) \wpr_0 (T_{-x,0} b)(y,\xi).
\end{aligned}
\end{equation}
For $x \mapsto u(x,\cdot) \in TP(\rr d, \mathscr S)$ this means that
\begin{equation}\nonumber
\begin{aligned}
A( a \wpr_0 b (y,D) ) u(x,y) & = (T_{-x,0} (a \wpr_0 b) )(y,D) u(x,\cdot)(y) \\
& = \left( (T_{-x,0} a) \wpr_0 (T_{-x,0} b) \right) (y,D) u(x,\cdot)(y) \\
& = (T_{-x,0} a)(y,D) (T_{-x,0} b) (y,D) u(x,\cdot)(y) \\
& = A (B u) (x,y).
\end{aligned}
\end{equation}
Since the map \eqref{representation3} preserves the identity operator, $a(y,D) \mapsto A(a(y,D))$ is a representation.

Suppose $A=0$ as an operator on $B^2 \otimes L^2$. Then $\mathscr M_x \| a_x(y,D) f \|_{L^2}^2 = 0$ for all $f \in \mathscr S$, which implies that $x \mapsto a_x(y,D) f=0$ in $C_{\rm ap}(\rr d,L^2)$. Thus $a_x(y,D) f=0$ for all $x \in \rr d$, in particular $x=0$, for any $f \in \mathscr S$. It follows that $a(y,D)=0$ in $L(L^2)$ and $L(B^2)$ and $a(y,D) \mapsto A(a(y,D))$ is faithful.

Next we look at the formal adjoint operation of $A$.
Since $a(y,D)^+ = a^+(y,D)$ where $a^+ \in APS_{\rho,\delta}^m$, $A(a(y,D)^+)$ extends to a bounded operator $B^2 \otimes H^s \mapsto B^2 \otimes H^{s-m}$ for any $s \in \ro$.
Let $g,f \in \mathscr S$ and $\la,\mu \in \rr d$ be the frequencies of two characters $e_\la$ and $e_\mu$. We have
\begin{equation}\nonumber
\begin{aligned}
(e_\la \otimes g,A^+ (e_\mu \otimes f) )_{B^2 \otimes L^2} & = (A (e_\la \otimes g),e_\mu \otimes f )_{B^2 \otimes L^2} \\
& = \mathscr M_x \left( e_\la(x) \ \overline{e_\mu(x)} \ (a_x(y,D)g,f)_{L^2} \right) \\
& = \mathscr M_x \left( e_\la(x) \ \overline{e_\mu(x)} \ (g,a_x(y,D)^+f)_{L^2} \right) \\
& = \mathscr M_x \left( (e_\la(x) \otimes g (\cdot) ,a_x(y,D)^+ e_\mu(x) \otimes f(\cdot))_{L^2} \right).
\end{aligned}
\end{equation}
Since $a_x(y,D)^+ = (T_{-x} a(y,D) T_x )^+ = T_{-x} a(y,D)^+ T_x$
and $\la \in \rr d$, $g \in \mathscr S$ are arbitrary,
we deduce that $A^+ = A(a(y,D)^+)$ when the operators act on the algebraic tensor product $TP(\rr d) \otimes \mathscr S(\rr d) \simeq TP(\rr d,\mathscr S(\rr d))$.
The representation $A$ is therefore $^+$-invariant.

Finally we prove the preservation of positivity property. Suppose $a(y,D) \geq 0$ on $TP(\rr d)$, which is equivalent to $a(y,D) \geq 0$ on $\mathscr S(\rr d)$. Since $a_x(y,D) \geq 0$ on $\mathscr S (\rr d)$ we have for $x \mapsto u(x,\cdot) \in TP(\rr d,\mathscr S)$
\begin{equation}\nonumber
( A u(x,\cdot), u (x,\cdot) )_{L^2} = ( a_x(y,D) u(x,\cdot), u(x,\cdot) )_{L^2} \geq 0, \quad x \in \rr d,
\end{equation}
which gives
\begin{equation}\nonumber
( A u, u )_{B^2 \otimes L^2} = \mathscr M_x ( A u(x,\cdot), u (x,\cdot) )_{L^2} \geq 0.
\end{equation}
Thus $A \geq 0$ on $TP(\rr d,\mathscr S)$.

On the other hand, suppose that $( A u, u )_{B^2 \otimes L^2} \geq 0$ for all $u \in TP(\rr d,\mathscr S)$.
Let $h \in C_{\rm ap}(\rr d)$ and define the sequence of functions $(K_n)_{n=1}^\infty$, used in the Bohr--Fourier reconstruction formula \eqref{fourierreconstruction1}, based on the frequencies $\La$ corresponding to $h$.
Then $\mathscr F_B^{-1} K_n \in TP(\rr d)$ is nonnegative \cite{Levitan1} and we may hence write
$\mathscr F_B^{-1} K_n = |f_n|^2$ for $f_n \in C_{\rm ap}(\rr d)$ (cf. \cite{Levitan1}).

Next we need Parseval's formula $(f,g)_B = (\mathscr F_B f, \mathscr F_B g)_{l^2}$, $f,g \in B^2(\rr d)$, which is the bilinear generalization of Plancherel's formula \eqref{plancherel1}.
For $p \in TP(\rr d)$ we obtain
\begin{equation}\nonumber
\begin{aligned}
& \left| \left( |p|^2, h \right)_B - \overline{h(0)} \right|
\leq \left| \left( |p|^2-|f_n|^2 , h \right)_B \right|
+ \left| \sum_{\la \in \La} K_n(\la) \overline{{\wh h}_\la} - \overline{h(0)} \right| \\
& \leq \| h \|_{L^\infty} ( \| p \|_{L^\infty} + \| f_n \|_{L^\infty})  \| p-f_n \|_{L^\infty}
+ \left| \sum_{\la \in \La} K_n(\la) {\wh h}_\la - h(0) \right|.
\end{aligned}
\end{equation}
From \eqref{fourierreconstruction1} it follows that the right hand side may be made arbitrarily small by first picking $n$ sufficiently large and then picking $p \in TP(\rr d)$ in order to make  $\| p-f_n \|_{L^\infty}$ as small as necessary. It follows that there exists a sequence $(p_n)_{n=1}^\infty \subseteq TP(\rr d)$, depending on $h$, such that
\begin{equation}\label{deltaconvsequence1}
\lim_{n \rightarrow +\infty} \left( |p_n|^2, h \right)_B = \overline{h(0)}.
\end{equation}
Let $\varphi \in \mathscr S$. From the proof of Proposition \ref{Acont1} we know that $h(x)=( a_x(y,D) \varphi, \varphi )_{L^2} \in C_{\rm ap}(\rr d)$.
Put $u_{n} = p_n \varphi \in TP(\rr d,\mathscr S)$ where the sequence $(p_n)_{n=1}^\infty \subseteq TP(\rr d)$ is chosen in order to satisfy \eqref{deltaconvsequence1}.
The assumption and \eqref{deltaconvsequence1} give
\begin{equation}\nonumber
\begin{aligned}
0 \leq \mathscr M_x (A u_n (x,\cdot), u_n (x,\cdot) )_{L^2}
& = \mathscr M_x \left( |p_n (x) |^2 ( a_x(y,D) \varphi, \varphi )_{L^2} \right) \\
& \rightarrow ( \varphi, a(y,D)  \varphi )_{L^2}, \quad n \rightarrow +\infty.
\end{aligned}
\end{equation}
It follows that $a(y,D) \geq 0$ on $\mathscr S$ as well as on $TP(\rr d)$.
\end{proof}

\section{Unitary equivalence for nonpositive order}

Also in this section we assume that \eqref{rhodeltakrav1} holds.
Let $\{\varphi_n \}_{n=0}^{\infty} \subseteq \mathscr S(\rr d)$ be an ONB for $L^2(\rr d)$.
We define the map $Q$ from an ONB in $B^2(\rr d) \otimes L^2(\rr d)$ to $L^2(\rr d, \l^2(\rr d))$ by
\begin{equation}\label{Qdef}
Q (e_{\la} \otimes \varphi_n) = \varphi_n (x) \ e_{-\la} (x) \ \delta_{(-\la)}, \quad \la \in \rr d, \quad n \in \no, \quad x \in \rr d.
\end{equation}
Initially $Q$ is defined only on the ONB $\{ e_{\la} \otimes \varphi_n \}_{\la \in \rr d,n \in \no}$. The following result extends its domain and range.
\begin{lem}
$Q: B^2(\rr d) \otimes L^2(\rr d) \mapsto L^2(\rr d, l^2(\rr d))$ defined by \eqref{Qdef} extends to a unitary transformation.
\end{lem}
\begin{proof}
Since $\{ e_\mu \otimes \varphi_n \}_{\mu \in \rr d, n \in \no}$ is an ONB for $B^2(\rr d) \otimes L^2(\rr d)$, we have
\begin{equation}\nonumber
\begin{aligned}
& ( Q (e_\mu \otimes \varphi_n), Q (e_{\mu'} \otimes \varphi_{n'}) )_{L^2(\rr d, l^2(\rr d) )} \\
& = ( \varphi_n  e_{-\mu},\varphi_{n'} e_{-\mu'})_{L^2(\rr d)} \ ( \delta_{(-\mu)},\delta_{(-\mu')})_{l^2(\rr d)}
= \delta_{(\mu-\mu')} \delta_{(n-n')},
\end{aligned}
\end{equation}
so $\{ Q (e_\mu \otimes \varphi_n) \}_{\mu \in \rr d, n \in \no}$ is an orthonormal set in $L^2(\rr d, l^2(\rr d))$. To prove that it is an ONB, suppose $F \in L^2(\rr d, l^2(\rr d))$ and
\begin{equation}\label{orthonormality2}
(F, Q (e_\mu \otimes \varphi_n) )_{L^2(\rr d,l^2(\rr d))} = 0 \quad \forall \mu \in \rr d \quad \forall n \in \no.
\end{equation}
By Pettis's measurability theorem \cite{Diestel1}, $F(x)$ takes values in a separable subset of $U \subseteq l^2(\rr d)$ for almost all $x \in \rr d$.
Thus there exists a null set $N \subseteq \rr d$ and a countable index set $\La \subseteq \rr d$
such that $x \in \rr d \setminus N$ and $\mu \in \rr d \setminus \La$ imply $(F(x), \delta_{(\mu)})_{l^2}=0$.
Denoting $F_{\mu} (x) = ( F(x), \delta_{(\mu)})_{l^2}$ for $\mu \in \rr d$, \eqref{orthonormality2} yields
\begin{equation}\nonumber
\left( F_{-\mu} e_\mu, \varphi_n \right)_{L^2(\rr d)} = 0 \quad \forall n \in \no
\quad \Longleftrightarrow \quad F_{-\mu} = 0 \quad \mbox{in} \quad L^2(\rr d),
\end{equation}
for any $\mu \in \rr d$.
Thus
$$
\| F \|_{L^2(\rr d,l^2)}^2 = \int_{\rr d \setminus N} \| F(x) \|_{l^2}^2 \ dx =  \int_{\rr d \setminus N} \sum_{\mu \in \La} |F_{\mu}(x)|^2 \ dx = 0,
$$
which implies that $\{ Q (e_\mu \otimes \varphi_n) \}_{\mu \in \rr d, n \in \no}$ is an ONB and $Q$ extends to a unitary transformation.
\end{proof}

\begin{lem}\label{fubinityp1}
Suppose $f \in C_b^\infty(\rr {d+n})$, $f(\cdot,y) \in C_{\rm ap}(\rr d)$ for all $y \in \rr n$ and $y \mapsto \| f(\cdot,y) \|_{L^\infty(\rr d)} \in L^1(\rr n)$.
Then $\int_{\rr n} f(\cdot,y) dy \in C_{\rm ap}(\rr d)$ and
\begin{equation}\nonumber
\mathscr M_x \left( \int_{\rr n} f(x,y) \ dy \right) = \int_{\rr n} \mathscr M_x( f(x,y) ) \ dy.
\end{equation}
\end{lem}
\begin{proof}
The integral $\int_{\rr n} f(x,y) dy$ can be approximated, uniformly in $x$, by a finite sum of $C_{\rm ap}(\rr d)$ functions. Therefore it belongs to $C_{\rm ap}(\rr d)$.
By Fubini's theorem we have
\begin{equation}\nonumber
\begin{aligned}
\mathscr M_x \left( \int_{\rr k} f(x,y) \ dy \right) & = \lim_{T \rightarrow + \infty} T^{-d}
\int_{K_T} \left( \int_{\rr n} f(x,y) \ dy \right) dx \\
& = \lim_{T \rightarrow + \infty}
\int_{\rr n} T^{-d} \left( \int_{K_T} f(x,y) \ dx \right) dy.
\end{aligned}
\end{equation}
The result follows from Lebesgue's dominated convergence theorem, since
the integrand with respect to $y$ is dominated by $y \mapsto \| f(\cdot,y) \|_{L^\infty(\rr d)}$.
\end{proof}

For $a \in APS_{\rho,\delta}^0$, Corollary \ref{sobolevcorollary1} implies $U(a)(D) \in \Lop(L^2(\rr d,l^2))$, and $A(a(y,D)) \in \Lop(B^2(\rr d) \otimes L^2(\rr d))$ by Proposition \ref{Acont1}.
The next result says that these operators are unitarily equivalent.

\begin{prop}\label{equivalent1}
If $a \in APS_{\rho,\delta}^0$ then the operators
\begin{equation}\nonumber
\begin{aligned}
U(a)(D): & \quad L^2(\rr d, l^2(\rr d)) \mapsto L^2(\rr d, l^2(\rr d)), \\
A(a(y,D)): & \quad B^2(\rr d) \otimes L^2(\rr d) \mapsto B^2(\rr d) \otimes L^2(\rr d)
\end{aligned}
\end{equation}
are unitarily equivalent.
\end{prop}
\begin{proof}
Let $\{\varphi_n \}_{n=0}^{\infty} \subseteq \mathscr S(\rr d)$ be an ONB for $L^2(\rr d)$,
and let $n \in \no$ and $\mu \in \rr d$.
We have
\begin{equation}\nonumber
\left( U(a)(\xi) \cdot \delta_{(-\mu)} \right)_\la = a_{-\mu - \la}(\xi+\mu), \quad \xi, \la \in \rr d,
\end{equation}
and hence for $y,\la \in \rr d$
\begin{equation}
\begin{aligned}
\left( U(a)(D) (Q e_\mu \otimes \varphi_n)(y) \right)_\la & = \int_{\rr d} e^{2 \pi i y \cdot \xi} \left( U(a)(\xi) \cdot \delta_{(-\mu)} \right)_\la \wh{\varphi_n e_{-\mu}}(\xi) \ d\xi \\
& = \int_{\rr d} e^{2 \pi i y \cdot (\xi-\mu)} \ a_{-\mu - \la}(\xi) \ \wh{\varphi_n}(\xi) \ d\xi \\
& = e_{-\mu} (y) \ a_{-\mu - \la} (D) \ \varphi_n (y). \label{mellanled0}
\end{aligned}
\end{equation}
On the other hand we have, writing $A=A(a(y,D))$,
\begin{equation}\nonumber
\begin{aligned}
A(e_\mu \otimes \varphi_n)(x,y) = e_\mu(x) \int_{\rr d} e^{2 \pi i y \cdot \xi} \ a(x+y,\xi) \ \wh \varphi_n(\xi) \ d \xi.
\end{aligned}
\end{equation}
We compute the coefficients of $A(e_\mu \otimes \varphi_n)$ with respect to $e_{\la'} \otimes \varphi_m$, $\la' \in \rr d$, $m \in \no$. By Lemma \ref{hilbertiso1} and Lemma \ref{fubinityp1} we have
\begin{equation}\nonumber
\begin{aligned}
& \left( A(e_\mu \otimes \varphi_n), e_{\la'} \otimes \varphi_m \right)_{B^2(\rr d) \otimes L^2(\rr d)} \\
& = \mathscr M_x \left( e_{-\la'}(x) \iint_{\rr {2d}} e^{2 \pi i (x \cdot \mu+y \cdot \xi)} \ a(x+y,\xi) \ \wh \varphi_n(\xi) \ \overline{\varphi_m(y)} \ d \xi \ dy \right) \\
& = \iint_{\rr {2d}} \mathscr M_x \left( e^{2 \pi i (x \cdot (\mu-\la')+y \cdot \xi)} \ a(x+y,\xi) \right) \ \wh \varphi_n(\xi) \ \overline{\varphi_m(y)} \ d \xi \ dy \\
& = \iint_{\rr {2d}} a_{\la'-\mu}(\xi) \ e^{2 \pi i (y \cdot (\la'-\mu)+y \cdot \xi)} \ \wh \varphi_n(\xi) \ \overline{\varphi_m(y)} \ d \xi \ dy \\
& = ( e_{\la'-\mu} \ a_{\la'-\mu}(D) \ \varphi_n, \varphi_m)_{L^2(\rr d)}.
\end{aligned}
\end{equation}
Since $A(e_\mu \otimes \varphi_n) \in B^2(\rr d) \otimes L^2(\rr d)$, we have
$$
\sum_{m=0}^\infty \left| \left( A(e_\mu \otimes \varphi_n), e_{\la'} \otimes \varphi_m \right)_{B^2 \otimes L^2} \right|^2 < \infty
$$
for each $\la' \in \rr d$. Therefore for any $\la' \in \rr d$
\begin{equation}
\begin{aligned}
& \sum_{m=0}^\infty \left( A(e_\mu \otimes \varphi_n), e_{\la'} \otimes \varphi_m \right)_{B^2 \otimes L^2} \varphi_m(y) \\
& = \sum_{m=0}^\infty ( e_{\la'-\mu} \ a_{\la'-\mu}(D) \ \varphi_n , \varphi_m)_{L^2} \ \varphi_m(y)
= e_{\la'-\mu} (y) \ a_{\la'-\mu}(D) \ \varphi_n (y)  \label{mellanled1}
\end{aligned}
\end{equation}
with convergence in $L^2(\rr d)$. The identity \eqref{mellanled1} thus holds for all $y \in \rr d \setminus N_{\mu,n,\la'}$
where $N_{\mu,n,\la'} \subseteq \rr d$ is a null set depending on $\mu,n,\la'$.
We obtain for any $\la \in \rr d$, and $y \in \rr d \setminus N_{\mu,n,-\la}$, using \eqref{Qdef}, \eqref{mellanled1} and \eqref{mellanled0},
\begin{equation}\nonumber
\begin{aligned}
& \left( Q A(e_\mu \otimes \varphi_n) (y) \right)_\la \\
& = Q \left( \sum_{\la' \in \rr d, \ m \in \no} \left( A(e_\mu \otimes \varphi_n), e_{\la'} \otimes \varphi_m \right)_{B^2 \otimes L^2} e_{\la'} \otimes \varphi_m \right) (y)_\la \\
& = \sum_{\la' \in \rr d, \ m \in \no} \left( A(e_\mu \otimes \varphi_n), e_{\la'} \otimes \varphi_m \right)_{B^2 \otimes L^2} \ Q (e_{\la'} \otimes \varphi_m)(y)_\la \\
& = \sum_{\la' \in \rr d, \ m \in \no} \left( A(e_\mu \otimes \varphi_n), e_{\la'} \otimes \varphi_m \right)_{B^2 \otimes L^2} \ \varphi_m(y) \ e_{-\la'}(y) \ \left( \delta_{(-\la')} \right)_\la \\
& = \sum_{m \in \no} \left( A(e_\mu \otimes \varphi_n), e_{-\la} \otimes \varphi_m \right)_{B^2 \otimes L^2} \ \varphi_m(y) \ e_{\la}(y) \\
& = e_{-\mu} (y) \ a_{-\la-\mu}(D) \ \varphi_n (y)
= \left( U(a)(D) \ (Q e_\mu \otimes \varphi_n)(y) \right)_\la.
\end{aligned}
\end{equation}
Since $U(a)(D) (Q e_\mu \otimes \varphi_n), \ Q A(e_\mu \otimes \varphi_n) \in L^2(\rr d,l^2)$ there exists by Pettis's measurability theorem
a null set $N_{\mu,n} \subseteq \rr d$ and a countable set $\La_{\mu,n} \subseteq \rr d$ such that $y \in \rr d \setminus N_{\mu,n}$ implies $(U(a)(D) (Q e_\mu \otimes \varphi_n) (y))_\la = ( Q A(e_\mu \otimes \varphi_n) (y) )_\la = 0$ for $\la \in \rr d \setminus \La_{\mu,n}$. Define the null set
$$
N = N_{\mu,n} \cup \bigcup_{\la \in \La_{\mu,n}} N_{\mu,n,-\la} \subseteq \rr d.
$$
We have
\begin{equation}\nonumber
\begin{aligned}
& \| Q A(e_\mu \otimes \varphi_n) - U(a)(D) (Q e_\mu \otimes \varphi_n) \|_{L^2(\rr d, l^2)}^2 \\
& = \int_{\rr d \setminus N} \sum_{\la \in \La_{\mu,n}} | Q A(e_\mu \otimes \varphi_n)(y)_\la - U(a)(D) (Q e_\mu \otimes \varphi_n)(y)_\la |^2 \ dy = 0.
\end{aligned}
\end{equation}
Therefore $A$ and $Q^* U(a)(D) Q$ are equal when they act on the ONB $\{ e_\mu \otimes \varphi_n \}_{\mu \in \rr d, n \in \no}$, and hence $A = Q^* U(a)(D) Q$ on the whole Hilbert space $B^2(\rr d) \otimes L^2(\rr d)$.
\end{proof}

\section{Applications to spectral theory}

A closed operator $T$ on a Hilbert space $H$ is called Fredholm if $\dim \Ker \ T < \infty$ and $\dim \Coker \ T < \infty$. Then $\Ran \ T = T \ \Dom \ T$ is automatically closed \cite[Thm.~I.3.2]{Edmunds1}.

We recall some facts about the spectrum of an unbounded, closed, densely defined operator $T$ on a Hilbert space $H$ (cf. \cite{Reed1}).
The resolvent set $\rho_H(T)$ consists of all $s \in \co$ such that $T - s I$ is injective, surjective and has a bounded inverse.
The spectrum is the complement $\sigma_H=\sigma_H (T) = \co \setminus \rho_H (T)$.
It is partitioned as $\sigma_H=\sigma_{H,d} \bigcup \sigma_{H,\rm ess}$ into the essential spectrum $\sigma_{H,\rm ess}$ and the discrete spectrum $\sigma_{H,d}= \sigma_H \setminus \sigma_{H,\rm ess}$.
There are several definitions of the essential spectrum (cf. \cite{Edmunds1,Reed1,Kato1}). We use the following: $s \in \sigma_{H,\rm ess}$ if and only if $T-sI$ is not a Fredholm operator.
If $T$ is selfadjoint we have the following characterization of $\sigma_{H,d}$ \cite[Thm.~IX.1.6]{Edmunds1}.
\begin{equation}\label{discreteselfadjoint1}
s \in \sigma_{H,d} \quad \mbox{if and only if $s$ is isolated in $\sigma_H$ and $\ \dim \Ker (T - sI) < \infty$}.
\end{equation}
Here $s$ isolated in $\sigma_H$ means that there exists an $\ep>0$ such that $(s-\ep,s+\ep) \cap \sigma_H=\{s\}$.

The spectrum is also partitioned as $\sigma_H = \sigma_{H,\rm p} \bigcup \sigma_{H,\rm cont} \bigcup \sigma_{H,\rm res}$ \cite{Edmunds1}. The point spectrum $\sigma_{H,\rm p}$ consists of all $s \in \co$ such that $T - s I$ is not injective,
the continuous spectrum $\sigma_{H,\rm cont}$ consists of all $s \in \co$ such that $T - s I$ is injective and
$\overline{\Ran (T - s I)} = H$, and
the residual spectrum $\sigma_{H,\rm res}$ is all $s \in \co$ such that $T - s I$ is injective and
$\overline{\Ran (T - s I)} \subsetneq H$.
Let $T$ be a densely defined closable, not necessarily closed, operator $T$ with closure $\overline T$.
Then $s \notin \sigma_H(\overline T)$ if and only if for some $C>0$
\begin{equation}\label{notspectrum}
\begin{aligned}
\| (T - s I) f \|_H & \geq C \| f \|_H, \quad f \in \Dom \ T, \\
\| (T - s I)^* f \|_H & \geq C \| f \|_H, \quad f \in \Dom \ T^*,
\end{aligned}
\end{equation}
where $T^*$ denotes the adjoint of $T$.

In \cite[Thm.~11.1]{Shubin2} Shubin proves the following result for a.p. pseudodifferential operators, considered as possibly unbounded operators on $L^2(\rr d)$. Here $\overline{a(x,D)}$ denotes the closure in $L^2(\rr d)$ of a closable operator $a(x,D)$ with original domain $C_c^\infty(\rr d)$.

\begin{prop}\label{shubinfredholm1}
(Shubin \cite{Shubin2})
Suppose that \eqref{rhodeltakrav1} holds and $a \in APS_{\rho,\delta}^0$,
or suppose that $0 \leq \delta < \rho \leq 1$ and $a \in APHS_{\rho,\delta}^{m,m_0}$ where $m \geq m_0>0$.
Let $s \in \co$.
If the operator $\overline{a(x,D)} - sI$ is Fredholm on $L^2(\rr d)$ then it has a bounded inverse.
\end{prop}

As a corollary it follows that the whole spectrum of certain a.p. pseudodifferential operators, considered as possibly unbounded operators on $L^2(\rr d)$, is essential.

\begin{cor}\label{shubinessential1}
(Shubin \cite{Shubin2})
Suppose that \eqref{rhodeltakrav1} holds and $a \in APS_{\rho,\delta}^0$,
or suppose that $0 \leq \delta < \rho \leq 1$ and $a \in APHS_{\rho,\delta}^{m,m_0}$ where $m \geq m_0>0$.
Then $\sigma_{L^2}(\overline{a(x,D)}) = \sigma_{L^2,\rm ess}(\overline{a(x,D)})$.
\end{cor}

The corollary gives the following result that nonzero operators with symbols in $APS_{\rho,\delta}^0$ cannot be compact on $L^2(\rr d)$. For related results, see \cite[p.~292]{Coburn1} and \cite[Cor.~5.2]{Rabinovich1}.

\begin{prop}\label{compactimplieszero1}
If \eqref{rhodeltakrav1} holds and $a \in APS_{\rho,\delta}^0$ then $a(x,D)$, considered as an operator on $L^2(\rr d)$, is not compact unless it is zero.
\end{prop}
\begin{proof}
Suppose $a(x,D)$ is compact as an operator on $L^2(\rr d)$. Then $b(x,D) = a(x,D)^* a(x,D)$ is compact and selfadjoint. The spectrum of $b(x,D)$ with respect to $L^2(\rr d)$ is a sequence of eigenvalues $\mu \in \ro$ such that $\dim \Ker (b(x,D)-\mu I) < \infty$ if $\mu \neq 0$, with the only possible limit point zero \cite[Thms.~I.1.7 and I.1.9]{Edmunds1}. But according to Corollary \ref{shubinessential1} the spectrum of $b(x,D)$ is essential. By \eqref{discreteselfadjoint1} the spectrum is thus $\{ 0 \}$. A selfadjoint operator $b(x,D)$ with spectrum equal to zero is zero, due to the formula (see \cite[Thm.~VI.6]{Reed1})
\begin{equation}\nonumber
\sup_{s \in \sigma(b(x,D))} |s| = \| b(x,D) \|_{\Lop(L^2)}.
\end{equation}
Hence $a(x,D)^* a(x,D) = 0$ which is equivalent to $a(x,D) = 0$.
\end{proof}

The next result concerns the statements corresponding to Proposition \ref{shubinfredholm1} and Corollary \ref{shubinessential1} when the operators act on $B^2(\rr d)$ instead of $L^2(\rr d)$. A brief sketch of a proof of the following proposition is given in \cite[pp.~189--190]{Rozenblum1}. We give a detailed proof based on the ideas in the proof of  \cite[Thm.~11.1]{Shubin2}.
Here $\overline{a(x,D)}$ denotes the closure in $B^2(\rr d)$ of a closable operator $a(x,D)$ with original domain $TP(\rr d)$.

\begin{prop}\label{besicovitchfredholm1}
Suppose that \eqref{rhodeltakrav1} holds and $a \in APS_{\rho,\delta}^0$,
or suppose that $0 \leq \delta < \rho \leq 1$ and $a \in APHS_{\rho,\delta}^{m,m_0}$ where $m \geq m_0>0$.
Let $s \in \co$.
If the operator $\overline{a(x,D)} - sI$ is Fredholm on $B^2(\rr d)$ then it has a bounded inverse.
\end{prop}
\begin{proof}
Set $A_s=\overline{a(x,D)} - sI$. We will prove the implication
\begin{equation}\label{propstatement1}
0 < \dim \Ker \ A_s < \infty \quad \Longrightarrow \quad \Ran \ A_s \quad \mbox{is not closed}.
\end{equation}
This implies the claimed result.
In fact, suppose \eqref{propstatement1} holds.
The symbol $a^+$ of the formal adjoint of $a(x,D)$ behaves as follows.
In the first case \eqref{rhodeltakrav1} and $a \in APS_{\rho,\delta}^0$, we have $a^+ \in APS_{\rho,\delta}^{m}$ \cite[Thm.~3.4]{Shubin1}.
In the second case $0 \leq \delta < \rho \leq 1$, $m \geq m_0>0$ and $a \in APHS_{\rho,\delta}^{m,m_0}$,
we have $a^+ \in APHS_{\rho,\delta}^{m,m_0}$ (cf. \cite[Prop.~I.5.3]{Shubin5}).
Denoting $A_s^+=\overline{a^+(x,D)} - sI$,
it thus it follows from \eqref{propstatement1}
\begin{equation}\label{propstatement2}
0 < \dim \Ker \ A_s^+ < \infty \quad \Longrightarrow \quad \Ran \ A_s^+ \quad \mbox{is not closed}.
\end{equation}
According to \cite[Thm.~4.2]{Shubin3} we have $A_s^+ = A_s^*$ when $a \in APHS_{\rho,\delta}^{m,m_0}$, $0 \leq \delta < \rho \leq 1$ and $m \geq m_0>0$.
(For \eqref{rhodeltakrav1} and $a \in APS_{\rho,\delta}^0$ this equality is trivial.)
Suppose $A_s$ is Fredholm. Then $\Ran \ A_s$ is closed, and by \cite[Thm.~IV.5.13]{Kato1} also
$\Ran \ A_s^*$ is closed.
The implication \eqref{propstatement1} gives
$$
\dim \Ker \ A_s = 0 \quad \mbox{or} \quad \dim \Ker \ A_s = \infty,
$$
and since $\Ker \ A_s^*= (\Ran \ A_s )^\perp$ \cite{Kato1},
\eqref{propstatement2} together with $A_s^+ = A_s^*$ give
$$
\dim \Coker \ A_s = 0 \quad \mbox{or} \quad \dim \Coker \ A_s = \infty.
$$
Hence $A_s$ cannot be Fredholm unless $\dim \Ker \ A_s = \dim \Coker \ A_s = 0$, which by the Closed Graph Theorem implies that $A_s$ has a bounded inverse.

It remains to show the implication \eqref{propstatement1}.
Suppose that $0 < M=\dim \Ker \ A_s < \infty$ and $\Ran \ A_s$ is closed. By \cite[Thm.~I.3.4]{Edmunds1} there exists $\ep>0$ such that
\begin{equation}\label{reductioadabsurdum1}
\| A_s \psi \|_B \geq \ep \| \psi \|_B, \quad \psi \in (\Ker \ A_s)^\perp.
\end{equation}
Let $N>M$ be an integer and let $\delta < \ep/N$.
Pick $g \in \Ker \ A_s$ that satisfies $\| g \|_B=1$ and $A_s g = 0$. Since $A_s$ is the closure of $a(x,D) - sI$ with domain $TP(\rr d)$ there exists $f \in TP(\rr d)$ such that $\| f \|_B=1$ and $\| A_s f \|_B < \delta/4$.

Let $\varphi_R \in C_c^\infty(\rr d)$ be a cutoff function parametrized by $R>0$ such that $\varphi_R(x)=1$ for $|x| \leq R$ and $\varphi_R(x)=0$ for $|x| \geq R+R^\kappa$ where $0<\kappa<1$, and let $|B_R|$ denote the Lebesgue measure of the ball $B_R \subseteq \rr d$ of radius $R$ centered at the origin. Then by \cite[Lemma~4.1]{Shubin4} we have
\begin{equation}\nonumber
\| A_s f \|_B^2 = \lim_{R \rightarrow +\infty} |B_R|^{-1} \| A_s (\varphi_R f) \|_{L^2}^2, \quad f \in TP(\rr d).
\end{equation}
Hence
$$
1 = \| f \|_B^2 = \lim_{R \rightarrow +\infty} |B_R|^{-1} \| \varphi_R f \|_{L^2}^2
$$
and we obtain by choosing
$h=f_R/\| f_R \|_{L^2}$, where $f_R = |B_R|^{-1/2} \varphi_R f$ and $R>0$ is sufficiently large, a function $h \in C_c^\infty(\rr d)$ such that $\| h \|_{L^2}=1$ and
\begin{equation}\label{smallnorm0}
\| A_s h \|_{L^2} < \delta/3.
\end{equation}
In the next step we construct, as in the proof of \cite[Thm.~11.1]{Shubin2}, $(h_k)_{k=1}^N \subseteq C_c^\infty(\rr d)$ such that
\begin{align}
(h_k,h_\ell)_{L^2} & = \delta_{(k-\ell)}, \quad 1 \leq k,\ell \leq N, \label{orthonormality1} \\
\| A_s h_k \|_{L^2} & < \delta/3, \quad \quad 1 \leq k \leq N. \label{smallnorm1}
\end{align}
The construction proceeds inductively starting from $N=1$ and $h_1=h$.
Supposing that \eqref{orthonormality1} and \eqref{smallnorm1} hold for $N=j-1$,
one defines $h_j=h(\cdot-y_j)$ for some translation parameter $y_j \in \rr d$. For sufficiently large $y_j$ the orthonormality \eqref{orthonormality1} is satisfied for $1 \leq k,\ell \leq j$. The bound \eqref{smallnorm1} for $k=j$
is obtained by means of \eqref{smallnorm0} and a (large) translation $y_j$ within a set of common $\theta$-almost periods of $\{ b_s(\cdot,\xi) \eabs{\xi}^{-2m'} \}_{\xi \in \rr d}$, where $m'>m$, $b_s$ is the symbol of $(a(x,D)-sI)^+ (a(x,D)-sI)$, and $\theta>0$ is sufficiently small.
For any $\theta>0$, the set of common $\theta$-almost periods for $\{ b_s(\cdot,\xi) \eabs{\xi}^{-2m'} \}_{\xi \in \rr d}$ is relatively dense, since $\{ b_s(\cdot,\xi) \eabs{\xi}^{-2m'} \}_{\xi \in \rr d}$ is precompact in $C_{\rm ap}(\rr d)$.

Now we use \cite[Lemmata 4.2 and 4.3]{Shubin4}. More precisely, we define the convolution
$$
\psi_j = \chi_j * h \in C_{\rm ap}^\infty(\rr d), \quad h \in C_c^\infty(\rr d), \quad \chi_j \in C_{\rm ap}^\infty(\rr d),
$$
where the sequence $(\chi_j)_{j=1}^\infty \subseteq C_{\rm ap}^\infty(\rr d)$ is chosen in such a way that
\begin{align}
g_j(y) & := \mathscr M_x ( \chi_j(x-y) \overline{\chi_j(x)} ) \rightarrow \delta_0 \quad \mbox{in} \quad \mathscr S'(\rr d), \quad j \rightarrow \infty, \nonumber \\
\| A_s h \|_{L^2} & = \lim_{j \rightarrow \infty} \| A_s \psi_j \|_B, \quad h \in C_c^\infty(\rr d). \label{smallnorm2}
\end{align}
Defining $\psi_{k,j} = \chi_j * h_k$ this gives, using Fubini's theorem,
\begin{equation}\nonumber
\begin{aligned}
(\psi_{k,j}, \psi_{\ell,j})_B = \iint_{\rr {2d}} h_k(y) \ \overline{h_\ell(z)} \ g_j(y-z) \ dy dz \longrightarrow (h_k,h_\ell)_{L^2}=\delta_{(k-\ell)},
\end{aligned}
\end{equation}
as $j \rightarrow \infty$, $1 \leq k,\ell \leq N$. The hermitian (Gramian) matrix $G_j=(\psi_{k,j}, \psi_{\ell,j})_B \in \mathbb C^{N \times N}$ is thus nonsingular for $j$ sufficiently large. It may be factorized as
$$
G_j = U_j \Sigma_j U_j^+
$$
where $U_j \in \mathbb C^{N \times N}$ is unitary and $\Sigma_j \in \mathbb C^{N \times N}$ is diagonal with diagonal elements $(\sigma_{j,k})_{k=1}^N$ equal to the nonzero eigenvalues of $G_j$. By Ger\v sgorin's theorem \cite{Horn1}, $\sigma_{j,k} \rightarrow 1$ as $j \rightarrow \infty$ for each $1 \leq k \leq N$.

We define
$$
\wt \psi_{k,j} = \sum_{n=1}^N (U_j^+)_{k,n} \psi_{n,j}, \quad \gamma_{k,j} = \wt \psi_{k,j}/\| \wt \psi_{k,j} \|_B, \quad 1 \leq k \leq N.
$$
Then $(\gamma_{k,j})_{k=1}^N \subseteq B^2 (\rr d)$ is an orthonormal system and $\| \wt \psi_{k,j} \|_B^2 = \sigma_{j,k}$ for $1 \leq k \leq N$.
By \eqref{smallnorm1} and \eqref{smallnorm2},  $\| A_s \psi_{n,j} \|_B < \delta/2$ for $1 \leq n \leq N$, for $j$ sufficiently large.
Therefore the fact that $U_j$ is unitary and $\sigma_{j,k} \rightarrow 1$ give
\begin{equation}\label{penultimastima1}
\| A_s \gamma_{k,j} \|_B \leq \sigma_{j,k}^{-1/2} \sum_{n=1}^N | (U_j^+)_{k,n}| \| A_s \psi_{n,j} \|_B
\leq \sigma_{j,k}^{-1/2} \delta/2 \sqrt{N} \leq \delta \sqrt{N}
\end{equation}
for all $1 \leq k \leq N$, provided $j$ is sufficiently large.

Let $\psi \in B^2(\rr d)$ belong to the subspace $V$ spanned by $(\gamma_{k,j})_{k=1}^N$. Then $\psi = \sum_{k=1}^N c_k \gamma_{k,j}$, $\| \psi \|_B^2 = \sum_{k=1}^N |c_k|^2$, and \eqref{penultimastima1} yields
\begin{equation}\label{ultimastima1}
\| A_s \psi \|_B \leq \sum_{k=1}^N |c_k| \| A_s \gamma_{k,j} \|_B \leq \delta N \| \psi \|_B < \ep \| \psi \|_B.
\end{equation}
According to the following Lemma \ref{hilbertsubspacelemma1} we may choose $\psi \in V \cap (\Ker \ A_s)^\perp$ such that $\psi \neq 0$. But this means that \eqref{ultimastima1} contradicts \eqref{reductioadabsurdum1}.
Therefore $0 < \dim \Ker \ A_s < \infty$ and $\Ran \ A_s$ is closed cannot hold, which proves the implication \eqref{propstatement1}.
\end{proof}

\begin{lem}\label{hilbertsubspacelemma1}
Let $H$ be a complex Hilbert space and let $U \subseteq H$ and $V \subseteq H$ be two nonzero finite-dimensional subspaces.
Then
$$
\dim(V) > \dim(U) \quad \Longrightarrow \quad V \cap U^\perp \neq 0.
$$
\end{lem}
\begin{proof}
We prove the equivalent implication
\begin{equation}\label{lemmastatement1}
V \cap U^\perp = 0 \quad \Longrightarrow \quad \dim(V) \leq \dim(U).
\end{equation}
Let $n=\dim(V)$ and let $(f_i)_{i=1}^n \subseteq V$ be an ONB for $V$. Let $P$ denote the orthogonal projection $P:H \mapsto U$. Then $P f_i \in U \setminus 0$ for each $1 \leq i \leq n$, since $P f_i=0$ gives the contradiction $f_i = (I-P)f_i \in V \cap U^\perp = 0$.
Suppose that $\sum_{i=1}^n a_i P f_i = 0$ where $a_i \in \co$ for $1 \leq i \leq n$. Then we have
$$
P \left( \sum_{i=1}^n a_i f_i \right) = 0,
$$
which implies $\sum_{i=1}^n a_i f_i \in V \cap U^\perp=0$. Thus $a_i=0$ for $1 \leq i \leq n$ which shows that
$(Pf_i)_{i=1}^n \subseteq U$ is a linearly independent set. Hence $\dim(U) \geq n = \dim(V)$.
\end{proof}

Proposition \ref{besicovitchfredholm1} has the following consequences, the second of which is proved analogously to Proposition \ref{compactimplieszero1}.

\begin{cor}\label{besicovitchessential1}
Suppose that \eqref{rhodeltakrav1} holds and $a \in APS_{\rho,\delta}^0$,
or suppose that $0 \leq \delta < \rho \leq 1$ and $a \in APHS_{\rho,\delta}^{m,m_0}$ where $m \geq m_0>0$.
Then $\sigma_{B^2}(\overline{a(x,D)}) = \sigma_{B^2,\rm ess}(\overline{a(x,D)})$.
\end{cor}

\begin{cor}\label{compactimplieszero2}
If \eqref{rhodeltakrav1} holds and $a \in APS_{\rho,\delta}^0$ then $a(x,D)$, considered as an operator on $B^2(\rr d)$, is not compact unless it is zero.
\end{cor}

Next we study the spectrum of operators whose symbols depend on only one variable, that is, the operator is either a Fourier multiplier or the operator of multiplication with a $C_{\rm ap}^\infty(\rr d)$ function.

\begin{prop}\label{singlevariabledependence1}
Let $m \in \ro$ and $a \in APS_{\rho,\delta}^m$.
\begin{enumerate}
\item[\rm(i)] If $a(x,\xi)=a(x)$ for all $x,\xi \in \rr d$ then
\begin{equation}\label{spectrumequality1}
\sigma_{B^2} ( a(x,D) ) = \sigma_{B^2,\rm ess} ( a(x,D) ) = \overline{\Ran (a)}.
\end{equation}
\item[\rm(ii)] If $a(x,\xi)=a(\xi) \in \ro$ for all $x,\xi \in \rr d$ then
\begin{align}
\sigma_{B^2} ( \overline{a(D)} ) & = \sigma_{B^2,\rm ess} ( \overline{a(D)} ) = \overline{ \Ran (a)}, \label{spectrumequality2a} \\
\sigma_{B^2,\rm p} ( \overline{a(D)} ) & = \Ran (a), \label{spectrumequality2b} \\
\sigma_{B^2,\rm cont} ( \overline{a(D)} ) & = \overline{ \Ran (a)} \setminus \Ran (a). \label{spectrumequality2c}
\end{align}
\end{enumerate}
\end{prop}
\begin{proof}
(i) In this case $a \in APS_{1,0}^0$ and $a(x,D) \in \Lop(B^2)$. The first equality of \eqref{spectrumequality1} follows from Corollary \ref{besicovitchessential1}.
The second equality is proved by means of the equivalence: $s \notin \sigma_{B^2}(a(x,D))$ if and only if \eqref{notspectrum} holds true for some $C>0$, with $T=a(x,D)$ and $H = \Dom \ a(x,D) = \Dom \ a(x,D)^* = B^2(\rr d)$.

In fact, if $s \notin \overline{\Ran (a)}$ then $\ep := \inf_{x \in \rr d} |a(x)-s| > 0$.
Hence for $f \in B^2(\rr d)$ we have
\begin{equation}\nonumber
\begin{aligned}
\| (a(x,D) - s I) f \|_B^2 & = \mathscr M_x ( | (a(x)-s)f(x)|^2 ) \geq \ep^2 \| f \|_B^2, \\
\| (a(x,D) - s I)^* f \|_B^2 & = \mathscr M_x ( | (\overline{a(x)-s)}f(x)|^2 ) \geq \ep^2 \| f \|_B^2,
\end{aligned}
\end{equation}
so $s \notin \sigma_{B^2}(a(x,D))$ follows from \eqref{notspectrum}. Thus $\sigma_{B^2} ( a(x,D) ) \subseteq \overline{\Ran (a)}$.

Suppose on the other hand $s \in \overline{\Ran (a)}$.
Let $\ep>0$. There exists $x_\ep \in \rr d$ such that $|a(x_\ep)-s| < \ep$.
The function $\rr d \ni x \mapsto a(x)-s$ may be extended to a continuous function on $\rrb d$.
Therefore there exists an open neighborhood $U_\ep \subseteq \rrb d$ containing $x_\ep$ such that $|a(x)-s| < \ep$ for $x \in U_\ep$.
The Haar measure $\mu(U_\ep)$ of $U_\ep$ must be nonzero, since otherwise we could cover the compact group $\rrb d$ with a finite number of translations of $U_\ep$, implying the contradiction that the Haar measure of $\rrb d$ is zero. The indicator function $\chi_{U_\ep}$ has thus nonzero $L^2(\rrb d)$ norm, and it follows that $\| (a(x,D) - s I) \chi_{U_\ep} \|_{B} \leq \ep \| \chi_{U_\ep} \|_B$. Since $\ep>0$ is arbitrary it follows that \eqref{notspectrum} cannot hold. Hence $\overline{\Ran (a)} \subseteq \sigma_{B^2} ( a(x,D) )$ and the second equality of \eqref{spectrumequality1} is proved.

(ii) In this case $a(D)$ is a Fourier multiplier operator and acts on $TP(\rr d)$ as $a(D)= \mathscr F_B^{-1} M_a \mathscr F_B$ where $(M_a x)_\la = a(\la) \ x_\la$, $\la \in \rr d$, $x \in l^2(\rr d)$.
We take the closure of $a(D)$ as an unbounded operator on $B^2(\rr d)$.
By the proofs of \cite[Thms.~III.9.2 and IX.6.2]{Edmunds1},
the closure is
\begin{equation}\label{unitaryequiv2}
\overline{a(D)}= \mathscr F_B^{-1} M_a \mathscr F_B
\end{equation}
with $\Dom \ \overline{a(D)}  = \{ f \in B^2(\rr d): \ \sum_{\la \in \rr d} |a(\la) \wh f_\la|^2 < \infty \}$, and $M_a$ is a closed operator with domain $\Dom \ M_a = \{ x \in l^2(\rr d): \ \sum_{\la \in \rr d} |a(\la) x_\la|^2 < \infty \}$. Moreover $M_a^\ast = M_{\overline a}$.

If $s \notin \overline{\Ran(a)}$, then $\ep=\inf_{\xi \in \rr d} | a(\xi)-s| >0$ and
\begin{equation}\nonumber
\begin{aligned}
\| (M_a - sI) z \|_{l^2} & \geq \ep \| z \|_{l^2}, \quad z \in \Dom \ M_a, \\
\| (M_a - sI)^* z \|_{l^2} & \geq \ep \| z \|_{l^2}, \quad z \in \Dom \ M_a^*.
\end{aligned}
\end{equation}
Hence $s \notin \sigma_{l^2} ( M_a )$, and thus $\sigma_{B^2}(\overline{a(D)}) = \sigma_{l^2} ( M_a ) \subseteq \overline{\Ran(a)}$, the first equality being a consequence of \eqref{unitaryequiv2}.

Suppose $s \in \overline{\Ran(a)}$. Let $\ep>0$ and pick $\xi \in \rr d$ such that $| a(\xi)-s| < \ep$.
We have
\begin{equation}\nonumber
\| (M_a - sI) \delta_{(\xi)} \|_{l^2} \leq \ep \| \delta_{(\xi)} \|_{l^2}, \quad \| (M_a - sI)^* \delta_{(\xi)} \|_{l^2} \leq \ep \| \delta_{(\xi)} \|_{l^2},
\end{equation}
and since $\ep>0$ is arbitrary we may conclude that $s \in \sigma_{l^2} ( M_a )$. Therefore $\sigma_{B^2}(\overline{a(D)}) = \overline{\Ran(a)}$.

Next we prove $\sigma_{B^2}(\overline{a(D)}) = \sigma_{B^2,\rm ess}(\overline{a(D)})$.
Suppose $s \in \sigma_{B^2}(\overline{a(D)}) \setminus \sigma_{B^2,\rm ess}(\overline{a(D)})$, so that $s \in \overline{\Ran(a)}$ and $M_a - sI$ is Fredholm on $l^2(\rr d)$. Then either $\Ker \ (M_a - sI) = 0$ or $\Ker \ (M_a - sI) = \linspan( \delta_{(\xi_1)},\dots ,\delta_{(\xi_n)} )$ where $a(\xi)=s$ $\Leftrightarrow$ $\xi \in \{ \xi_j \}_{j=1}^n$. Moreover, we have
\begin{equation}\label{lowerboundassumption1}
\| (M_a - sI) \ z \|_{l^2}^2
\geq C^2 \inf_{x \in \Ker \ (M_a - sI)} \| z-x \|_{l^2}^2
= C^2 \sum_{\xi \in \rr d \setminus \{ \xi_j \}_{j=1}^n } | z_\xi|^2
\end{equation}
for some $C>0$, for all $z \in \Dom \ (M_a - sI)$, according to \cite[Thm.~I.3.4]{Edmunds1}.
Because $a$ is continuous it is possible to pick $\xi \in \rr d \setminus \{ \xi_j \}_{j=1}^n$ such that $|a(\xi)-s|<C$. Then $z=\delta_{(\xi)}$ gives a contradiction to \eqref{lowerboundassumption1}. It follows that $\sigma_{B^2}(\overline{a(D)}) = \sigma_{B^2, \rm ess}(\overline{a(D)})$. Thus \eqref{spectrumequality2a} is proved and it remains to show \eqref{spectrumequality2b} and \eqref{spectrumequality2c}.

If $s=a(\xi)$ then $z=\delta_{(\xi)}$ belongs to the kernel of $M_a - sI$ and hence $e_\xi$ belongs to the kernel of $\overline{a(D)}-sI$, which proves $\Ran (a) \subseteq \sigma_{B^2,\rm p} ( \overline{a(D)} )$.
If $z \in l^2(\rr d) \setminus 0$ belongs to the kernel of $M_a - sI$ then $(a(\xi)-s) \ z_\xi = 0$ for all $\xi \in \rr d$, and thus $a(\xi)=s$ must hold for some $\xi \in \rr d$. Hence $\sigma_{B^2,\rm p} ( \overline{a(D)} ) \subseteq \Ran (a)$
which confirms \eqref{spectrumequality2b}.

If $\overline{ \Ran (a)} \setminus \Ran (a) = \emptyset$ then $\sigma_{B^2,\rm cont} = \emptyset$ also.
Suppose $\overline{ \Ran (a)} \setminus \Ran (a) \neq \emptyset$ and let $s \in \overline{ \Ran (a)} \setminus \Ran (a)$. Then $a(\xi) \neq s$ for all $\xi \in \rr d$ and thus $0=\Ker \ (M_a - sI) = ( \Ran (M_a - sI)^* )^\perp$.
Hence $\Ran (M_a - sI)^* = \Ran (M_{\overline a} - \overline s I)$ is dense in $l^2(\rr d)$, which is equivalent to
$\Ran (M_a - sI)$ being dense in $l^2(\rr d)$.
Thus $\overline{ \Ran (a)} \setminus \Ran (a) \subseteq \sigma_{B^2,\rm cont} ( \overline{a(D)} )$.
Finally $\sigma_{B^2,\rm cont} ( \overline{a(D)} ) \subseteq \sigma_{B^2} ( \overline{a(D)} ) \setminus \sigma_{B^2,\rm p} ( \overline{a(D)} ) = \overline{ \Ran (a)} \setminus \Ran (a)$ which proves \eqref{spectrumequality2c}.
\end{proof}

\begin{rem}
It is interesting to compare Proposition \ref{singlevariabledependence1} (ii) with \cite[Thm.~IX.6.2]{Edmunds1}. The latter result says that if $a(D)$ is a partial differential operator with constant coefficients
then the spectrum of $\overline{a(D)}$ acting on $L^2(\rr d)$ is
$\sigma_{L^2} ( \overline{a(D)} ) = \sigma_{L^2,\rm ess} ( \overline{a(D)} ) = \sigma_{L^2,\rm cont} ( \overline{a(D)} = \overline{ \Ran (a)}$, and thus
$\sigma_{L^2,\rm p} ( \overline{a(D)} ) = \sigma_{L^2,\rm res} ( \overline{a(D)} ) = \emptyset$.
\end{rem}

In the next theorem we use the unitary equivalence \eqref{unitaryequiv1} between $a(x,D)$ acting on $B^2(\rr d)$ and $U(a)(0)$ acting on $l^2(\rr d)$.

\begin{prop}\label{essentialspectrum1}
Suppose \eqref{rhodeltakrav1} holds, $m \in \ro$ and $a \in APS_{\rho,\delta}^m$.
If there exists $\xi_0 \in \rr d$ such that $a(x,\xi_0) = \mathscr M (a(\cdot,\xi_0))$ for all $x \in \rr d$,
then $\mathscr M (a(\cdot,\xi_0)) \in \sigma_{B^2,\rm ess} ( \overline{a(x,D)} )$.
\end{prop}
\begin{proof}
Note that $a_0(\xi_0) = \mathscr M (a(\cdot,\xi_0))$ by \eqref{bohrfourier1}.
Plancherel's formula \eqref{plancherel1} yields for any $\xi \in \rr d$
\begin{equation}\label{limitapproachingzero1}
\begin{aligned}
\| ( U(a)(0) - sI) \ \delta_{(-\xi)} \|_{l^2}^2 & = \sum_{\la \in \rr d} \left|  a_{-\xi-\la} (\xi) - s \ \left( \delta_{(-\xi)} \right)_\la  \right|^2 \\
& = |a_0(\xi)-s|^2 + \sum_{\la \neq 0} | a_\la (\xi) |^2 \\
& = |a_0(\xi)-s|^2 + \| a(\cdot,\xi) \|_B^2 - |a_0(\xi)|^2.
\end{aligned}
\end{equation}
By the mean value theorem we have for $\xi,\eta \in \rr d$
\begin{equation}\nonumber
a(x,\xi+\eta)-a(x,\xi) = \left( \nabla_2 \re \ a(x,\xi+\theta_1 \eta) + i \nabla_2 \im \ a(x,\xi+\theta_2 \eta)\right) \cdot \eta
\end{equation}
where $\nabla_2$ denotes the gradient in the second $\rr d$ variable
and $0 \leq \theta_1,\theta_2 \leq 1$. It follows from this and \eqref{hormclass1} that for any $\xi \in \rr d$ we have
\begin{equation}\nonumber
\| a(\cdot,\xi+\eta)-a(\cdot,\xi) \|_{L^\infty} \rightarrow 0, \quad \eta \rightarrow 0,
\end{equation}
which implies
\begin{equation}\nonumber
\left| \ \| a(\cdot,\xi+\eta) \|_B - \| a(\cdot,\xi) \|_{B} \right| \rightarrow 0, \quad \eta \rightarrow 0.
\end{equation}
In particular $ \| a(\cdot,\xi_0+\eta) \|_B \rightarrow \| a(\cdot,\xi_0) \|_B = |a_0(\xi_0)|$ as $\eta \rightarrow 0$. If we pick a sequence $(\xi_j)_{j=1}^\infty$ of distinct vectors in $\rr d$ such that $\xi_j \rightarrow \xi_0$, we thus obtain from \eqref{limitapproachingzero1} with $s=a_0(\xi_0)$ and the continuity of $\xi \mapsto a_0(\xi)$ (cf. \cite{Wahlberg1})
\begin{equation}\nonumber
\begin{aligned}
\| ( U(a)(0) - a_0(\xi_0) ) I) \ \delta_{(-\xi_j)} \|_{l^2}^2 = & \ |a_0(\xi_j)-a_0(\xi_0)|^2 + \| a(\cdot,\xi_j) \|_B^2 - |a_0(\xi_j)|^2 \\
& \longrightarrow 0, \quad j \rightarrow \infty.
\end{aligned}
\end{equation}
Finally \eqref{unitaryequiv1} gives
\begin{equation}\label{singularsequence1}
\begin{aligned}
\| (a(x,D) - a_0(\xi_0) I) \ e_{\xi_j} \|_B
= & \ \| (\mathscr F_B R )^*\ (U(a)(0) - a_0(\xi_0) I) \ \mathscr F_B R \ e_{\xi_j} \|_B \\
= & \ \| (U(a)(0) - a_0(\xi_0) I) \ \delta_{(-\xi_j)} \|_{l^2} \\
& \longrightarrow 0, \quad j \rightarrow \infty.
\end{aligned}
\end{equation}
This means that the first inequality of \eqref{notspectrum} does not hold for any $C>0$.
Therefore $a_0(\xi_0) \in \sigma_{B^2}(\overline{a(x,D)})$.

Furthermore, we may use Weyl's criterion (see \cite[Thm.~IX.1.3]{Edmunds1} and \cite{Reed1} for the case of bounded operators).
In fact, the sequence $(e_{\xi_j})_{j=1}^\infty \subseteq B^2(\rr d)$ is orthonormal and therefore converges weakly to zero. Thus \eqref{singularsequence1} combined with \cite[Thm.~IX.1.3]{Edmunds1}  imply that $a_0(\xi_0) \in \sigma_{B^2,\rm ess} ( \overline{a(x,D)} )$.
\end{proof}

Our final result concerns the invariance of the spectra of an operator and its representations.

\begin{prop}\label{spectruminvariant1}
The spectrum is invariant as follows.

\begin{enumerate}
\item[\rm(i)] If \eqref{rhodeltakrav1} holds and $a \in APS_{\rho,\delta}^0$ then
\begin{equation}\nonumber
\begin{aligned}
\sigma_{L^2}(a(y,D)) & = \sigma_{B^2}(a(y,D)) = \sigma_{l^2}(U(a)(0)) \\
& = \sigma_{B^2 \otimes L^2}(A(a(y,D))) = \sigma_{L^2(\rr d, l^2)}(U(a)(D)).
\end{aligned}
\end{equation}

\item[\rm(ii)] If $0 \leq \delta < \rho \leq 1$ and $a \in APHS_{\rho,\delta}^{m,m_0}$ where $m \geq m_0>0$ then
\begin{equation}\nonumber
\begin{aligned}
\sigma_{L^2}(\overline{a(x,D)}) & = \sigma_{B^2}(\overline{a(x,D)}) = \sigma_{l^2}(\overline{U(a)(0)}) \\
& = \sigma_{B^2 \otimes L^2}(\overline{A(a(y,D))}).
\end{aligned}
\end{equation}
\end{enumerate}
\end{prop}
\begin{proof}
\rm(i) We use: $s \notin \sigma_{B^2}(a(x,D))$
if and only if there is a $C>0$ such that
\eqref{notspectrum} holds with $T=a(x,D)$, $H=B^2(\rr d)$ and $\Dom \ a(x,D) = \Dom \ a(x,D)^* = TP(\rr d)$.
Note that, for bounded operators, the adjoint equals the closure of the formal adjoint.
Thus \eqref{notspectrum} may be formulated equivalently as
\begin{equation}\label{notspectrum1}
\begin{aligned}
(a(x,D) - sI)^+ (a(x,D) - sI) - C^2 \geq 0, \\
(a(x,D) - sI) (a(x,D) - sI)^+ - C^2 \geq 0
\end{aligned}
\end{equation}
on $TP(\rr d)$.
By \eqref{symbolproduct1}, \eqref{matrixpositivity1}, \eqref{onequant1} and Corollary \ref{unitaryequivalent1}
this is equivalent to the positivity on $l_f^2$
\begin{equation}\nonumber
\begin{aligned}
(U(a)(0) - sI)^+ (U(a)(0) - sI) - C^2 & \geq 0, \\
(U(a)(0) - sI) (U(a)(0) - sI)^+ - C^2 & \geq 0.
\end{aligned}
\end{equation}
Moreover, \eqref{notspectrum1} is by Corollary \ref{representation1} equivalent to the positivity on $\mathscr S(\rr d,l_f^2)$
\begin{equation}\nonumber
\begin{aligned}
(U(a)(D) - sI)^+ (U(a)(D) - sI) - C^2 & \geq 0 \\
(U(a)(D) - sI) (U(a)(D) - sI)^+ - C^2 & \geq 0.
\end{aligned}
\end{equation}
Finally, \eqref{notspectrum1} is equivalent to the positivity on $TP(\rr d,\mathscr S(\rr d))$
\begin{equation}\nonumber
\begin{aligned}
( A(a(y,D)) - sI)^+ (A(a(y,D)) - sI) - C^2 & \geq 0 \\
(A(a(y,D)) - sI) (A(a(y,D)) - sI)^+ - C^2 & \geq 0
\end{aligned}
\end{equation}
according to Proposition \ref{representation2}.
The equivalences
\begin{equation}\nonumber
\begin{aligned}
s \in \sigma_{B^2}(a(x,D)) & \ \Leftrightarrow  \ s \in \sigma_{l^2}(U(a)(0)) \ \Leftrightarrow \ s \in \sigma_{L^2(\rr d, l^2)}(U(a)(D)) \\
& \ \Leftrightarrow  \ s \in \sigma_{B^2 \otimes L^2}( A(a(y,D)))
\end{aligned}
\end{equation}
follow.
Finally, if $0 \leq \delta < \rho \leq 1$ then $\sigma_{L^2}(a(x,D)) = \sigma_{B^2}(a(x,D))$ is \cite[Thm.~5.1]{Shubin3}, and the proof extends to the assumption \eqref{rhodeltakrav1}.
For an alternative proof of $\sigma_{L^2}(a(x,D)) = \sigma_{B^2 \otimes L^2}(A(a(y,D)))$ we refer to \cite{Shubin4b}.

\rm(ii) For unbounded operators the closure of the formal adjoint $\overline{a(x,D)^+}$ differs in general from the adjoint $a(x,D)^*$.
However, here the assumptions imply $\overline{a(x,D)^+} = (\overline{a(x,D)})^*$ (cf. the proof of Proposition \ref{besicovitchfredholm1} and \cite{Shubin3}),
that is, the closure of the formal adjoint equals the adjoint of the closure.
(We note that generally $(\overline{T})^*=T^*$ for a closable operator $T$ \cite[Thm.~VIII.1]{Reed1}.)
Thus \eqref{notspectrum1} is equivalent to $s \notin \sigma_{B^2}(\overline{a(x,D)})$.
By \eqref{unitaryequiv1} and \cite[Thm.~13.2]{Rudin1} we have
\begin{equation}\nonumber
\begin{aligned}
\overline{U(a)(0)^+} & = \mathscr F_B R \ \overline{a(x,D)^+} \ (\mathscr F_B R)^* \\
& = \mathscr F_B R \ (\overline{a(x,D)})^* \ (\mathscr F_B R)^* \\
& = (\overline{U(a)(0)})^*.
\end{aligned}
\end{equation}
The arguments used in the proof of \rm(i) now proves that $\sigma_{B^2}(\overline{a(y,D)}) = \sigma_{l^2}(\overline{U(a)(0)})$. The identity $\sigma_{L^2}(\overline{a(y,D)}) = \sigma_{B^2}(\overline{a(y,D)})$ is \cite[Thm.~5.2]{Shubin3}.

Finally we prove $\sigma_{B^2}(\overline{a(x,D)}) = \sigma_{B^2 \otimes L^2}(\overline{A})$ where $A=A(a(y,D))$.
The arguments in the proof of (i) are valid for this purpose, provided that we can show that $\overline{(A^+)} = (\overline{A})^*$. According to \cite[Thm.~4.1]{Shubin3} it suffices to prove that $A^+ A$ is essentially selfadjoint, which means that its closure is selfadjoint.
By \cite[Thm.~VIII.3]{Reed1} this is equivalent to
\begin{equation}\label{injective1}
\Ker ( (A^+ A)^* \pm i I) = \{ 0 \}.
\end{equation}
To prove the theorem it thus suffices to establish \eqref{injective1}.
We know that $a(y,D)^+ \in APHL_{\rho,\delta}^{m,m_0}$, $a(y,D)^+ a(y,D) \in APHL_{\rho,\delta}^{2m,2m_0}$ and
\begin{equation}\nonumber
a(y,D)^+ a(y,D) +iI \in APHL_{\rho,\delta}^{2m,2m_0}
\end{equation}
(cf. \cite[Props.~I.5.2, I.5.3 and Lemma I.5.3]{Shubin5}).
According to \eqref{regularizer1} there exists a symbol $b \in APHS_{\rho,\delta}^{-2m_0,-2m}$ such that
\begin{equation}\label{parametrix1}
b(y,D) \left( a(y,D)^+ a(y,D) +iI \right) = I - T \quad \mbox{where $T \in APL^{-\infty}$}.
\end{equation}
The operator $A^+A$ has domain $\Dom (A^+A) = TP(\rr d, \mathscr S(\rr d))$ and is symmetric.
We may extend the domain and define the operator $B=A^+A$ with domain $\Dom (B) = \bigcap_{s \in \ro} B^2 \otimes H^s$. Then $A^+A \subseteq B$ and $B$ is symmetric.

Suppose that $f \in \Dom (A^+ A)^*  \subseteq B^2 \otimes L^2$ and
\begin{equation}\label{kernelelement1}
((A^+ A)^* + i I)f = 0.
\end{equation}
This is equivalent to
\begin{equation}\nonumber
0 = ( ( (A^+ A)^* + i I)f,g )_{B^2 \otimes L^2} = (f, ( A^+ A -iI ) g)_{B^2 \otimes L^2}
\end{equation}
for all $g \in TP(\rr d, \mathscr S)$.
Let $(f_n)_{n=1}^\infty \subseteq TP(\rr d, \mathscr S)$ be a sequence such that $f_n \rightarrow f$ in $B^2 \otimes L^2$.
By Proposition \ref{Acont1}, $( A^+ A +iI ) f_n \rightarrow ( A^+ A +iI ) f$ in $B^2 \otimes H^{-2m}$.
For $F \in B^2 \otimes H^{-2m}$ and $G \in B^2 \otimes H^{2m}$
we have the inequality
$$
\left| \left( F,G \right)_{B^2 \otimes L^2} \right| \leq \| F \|_{B^2 \otimes H^{-2m}} \| G \|_{B^2 \otimes H^{2m}}.
$$
Thus the symmetry of $A^+A$ gives for any $g \in TP(\rr d, \mathscr S)$
\begin{equation}\nonumber
\begin{aligned}
(f, ( A^+ A -iI ) g)_{B^2 \otimes L^2}
& = \lim_{n \rightarrow \infty} (f_n, ( A^+ A -iI ) g)_{B^2 \otimes L^2} \\
& = \lim_{n \rightarrow \infty} ( ( A^+ A +iI ) f_n, g)_{B^2 \otimes L^2} \\
& = ( ( A^+ A +iI ) f,g)_{B^2 \otimes L^2}.
\end{aligned}
\end{equation}
Hence $(A^+ A + i I)f = 0$ in $B^2 \otimes L^2$. If we now appeal to Proposition \ref{representation2} and use \eqref{parametrix1} we get
\begin{equation}\nonumber
\begin{aligned}
0 & = A(b(y,D)) (A^+ A + i I)f \\
& = A(b(y,D)) A( a(y,D)^+ a(y,D) +iI ) f \\
& = A \left( b(y,D) (a(y,D)^+ a(y,D) +iI) \right) f \\
& = A( I - T ) f \\
& = f - A(T) f.
\end{aligned}
\end{equation}
Since $A(T): B^2 \otimes L^2 \mapsto \bigcap_{s \in \ro} B^2 \otimes H^s$ by Proposition \ref{Acont1}, and we know a priori that $f \in B^2 \otimes L^2$, we may conclude that actually $f \in \bigcap_{s \in \ro} B^2 \otimes H^s$. This means that $f \in \Dom ( B + i I )$, and $(B + i I)f = 0$ has only the trivial solution $f=0$ since $B$ is symmetric \cite[Thm.~13.16]{Rudin1}. By \eqref{kernelelement1} we have now proved the ``$+$'' case of \eqref{injective1}, and the ``$-$'' case follows similarly.
\end{proof}

\begin{rem}
Comparing case (i) and (ii) of Proposition \ref{spectruminvariant1}, the spectrum $\sigma_{L^2(\rr d, l^2)}(\overline{U(a)(D)})$ is conspicuously missing in case (ii). However,
for $0 \leq \delta < \rho \leq 1$, $m \geq m_0>0$ and $a \in APHS_{\rho,\delta}^{m,m_0}$ it seems difficult to prove that $\sigma_{B^2}(\overline{a(x,D)}) = \sigma_{L^2(\rr d, l^2)}(\overline{U(a)(D)})$. The reason is that Corollary \ref{sobolevcorollary1} gives, for negative $m$, only $U(a)(D): L^2(\rr d, l^2) \mapsto H^{m}(\rr d, l_{-m}^2)$, whereas we need $H^{-m}(\rr d, l_{-m}^2)$, that is a gain of regularity, in the right hand side to prove equality of the spectra, using techniques similar to the last part of the proof of Proposition \ref{spectruminvariant1} (cf. \cite{Shubin2,Shubin3}).

The fact that $\sigma_{B^2}(\overline{a(x,D)}) = \sigma_{B^2 \otimes L^2}(\overline{A(a(y,D))})$ under the same assumptions is an advantage of the representation $A(a(x,D))$ compared to the representation $U(a)(D)$. (See the next section.)
\end{rem}

\section{Remarks on representations in a factor of type II$_{\infty}$}

Let $\mathscr A$ be a von Neumann algebra (cf. \cite{Dixmier1,Kadison1}), that is, an algebra of bounded operators on a Hilbert space $H$, which is closed under the adjoint operation, and $\mathscr A=\mathscr A''$ where $\mathscr A'=\{ B \in \Lop(H): BA=AB \ \forall A \in \mathscr A \}$ denotes the commutator of $\mathscr A$.

A closed operator $T$ on a Hilbert space $H$ is said to be adjoined (or affiliated \cite{Kadison1}) to $\mathscr A$, denoted $T \eta \mathscr A$, if $BT \subseteq TB$ for all $B \in \mathscr A'$. This is equivalent to $(T-sI)^{-1} \in \mathscr A$ for all $s \notin \sigma_H(T)$ provided $\co \setminus \sigma_H(T) \neq \emptyset$, and if $T$ is selfadjoint then $T \eta \mathscr A$ is equivalent to $P_t \in \mathscr A$ for all spectral projections $P_t$, $t \in \ro$ (cf. \cite[Prop.~7.1]{Shubin2}). If $T$ is bounded then $T \eta \mathscr A$ means $T \in \mathscr A$.

In \cite{Coburn1,Shubin2,Shubin4b} it is shown that the representation $A=A(a(y,D))$ maps $APL_{\rho,\delta}^\infty$ into an algebra of operators that is associated with a certain von Neumann algebra $\mathscr A_B$. That is, we have $A(a(y,D)) \ \eta \ \mathscr A_B$ for all $a \in APS_{\rho,\delta}^\infty$. The von Neumann algebra $\mathscr A_B$ is generated by the family of operators $\{ M_\la \otimes M_\la, \ I \otimes T_\mu \}_{\la,\mu \in \rr d}$ acting on $B^2(\rr d) \otimes L^2(\rr d)$, and its commutant $\mathscr A_B'$ is generated by the family of operators $\{ T_{-\la} \otimes T_\la, \ M_\mu \otimes I \}_{\la,\mu \in \rr d}$ (cf. \cite{Dixmier1}).

The algebra $\mathscr A_B$ is a so-called \emph{factor}, meaning that $\mathscr A_B \cap \mathscr A_B' = \co I$. Furthermore, $\mathscr A_B$ is of type II$_{\infty}$ which means that there exists a faithful normal semifinite trace (cf. \cite{Kadison1,Shubin4b}) which takes values in $[0,+\infty]$ on the space of orthogonal projections in $\mathscr A_B$, and which is unique up to scalar multiplication. If $A$ is selfadjoint one defines the function $N(t) = {\rm tr} \ P_t$, where $\{ P_t \}_{t \in \ro}$ denotes the family of spectral projections, which characterizes the distribution of the spectrum of the operator $A$.

For a uniformly elliptic, essentially selfadjoint partial differential operator $a(y,D)$ with coefficients in $C_{\rm ap}^\infty(\rr d)$, it is possible to show results about the asymptotic behavior of $N(t)$ as $t \rightarrow +\infty$, and to give upper bounds on the lengths of lacunae (gaps) in the spectrum of $\overline{A(a(y,D))}$. Such results have been obtained by Shubin \cite{Shubin2}, using the representation $A(a(y,D))$. According to Proposition \ref{spectruminvariant1} (ii) we have $\sigma_{L^2}(\overline{a(x,D)}) = \sigma_{B^2}(\overline{a(x,D)}) = \sigma_{B^2 \otimes L^2}(\overline{A(a(x,D))})$, so these spectral results apply to the operator $\overline{a(x,D)}$ acting on $L^2$ and on $B^2$.

Due to the equivalence of the representations $A$ and $\wt U$ (Proposition \ref{equivalent1}), also the representation $\wt U (a(y,D))$ is adjoined to a von Neumann algebra of type II$_\infty$, provided $m \leq 0$. In fact, we have $\wt U (a(y,D)) \ \eta \ Q \mathscr A_B Q^*$ for all $a \in APL_{\rho,\delta}^0$, where $Q \mathscr A_B Q^*$ is a factor of type II$_{\infty}$, conjugate to $\mathscr A_B$. However, the representation $\wt U (a(y,D))$ does not seem to be as useful as $A(a(y,D))$ in order to obtain results for the spectrum of $\overline{a(y,D)}$ similar to Shubin's results described in the preceding paragraph, 
since for $m>0$, 
(i) $\wt U (a(y,D))$ is not adjoined to a factor of type II$_{\infty}$, and  
(ii) Proposition \ref{spectruminvariant1} gives no connection between $\sigma_{L^2}(\overline{a(x,D)})$ and $\sigma_{L^2(\rr d, l^2)}(\overline{U(a)(D)})$.


\begin{thebibliography}{2000}

\bibitem{Amerio1} \textsc{Amerio~L. and Prouse~G.}, \textit{Almost-Periodic Functions and Functional Equations}, Van Nostrand, 1971.

\bibitem{Amann1} \textsc{Amann~H.}, \textit{Vector-Valued Distributions and Fourier Multipliers}, unpublished manuscript,
2003.

\bibitem{Coburn1} \textsc{Coburn~L.~A., Moyer~R.~D. and Singer~I.~M.}, \textit{$C^*$-algebras of almost periodic pseudo-differential operators},
Acta. Math. \textbf{139} (1973), 279 -- 307.

\bibitem{Corduneanu1} \textsc{Corduneanu~C.}, \textit{Almost Periodic Functions}, Interscience Publishers, 1968.

\bibitem{Dedik1} \textsc{Dedik~P.~E.}, \textit{Theorems on the boundedness of almost-periodic pseudodifferential operators}, Siberian Math. J. \textbf{22} (3) (1981), 361--369.

\bibitem{Diestel1} \textsc{Diestel~J. and Uhl~J.~J.~Jr}, \textit{Vector Measures}, Math Surveys \textbf{15}, AMS, 1977.

\bibitem{Dixmier1} \textsc{Dixmier~J.}, \textit{Von Neumann Algebras}, North-Holland, 1981.

\bibitem{Filippov1} \textsc{Filippov~O.~E.}, \textit{Reduction of operators with almost periodic symbols to operators over $C^*$-algebras on sections of associated bundles over a torus}, Ann. Glob. Anal. Geom. \textbf{8} (2) (1990), 113--126.

\bibitem{Edmunds1} \textsc{Edmunds~D.~E. and Evans~W.~D.}, \textit{Spectral Theory and Differential Operators}, Oxford Science Publications, 1987.

\bibitem{Folland1} \textsc{Folland~G.~B.}, \textit{Harmonic Analysis in Phase Space}, Princeton University Press,
1989.

\bibitem{Gladyshev1} \textsc{Gladyshev~E.}, \textit{Periodically and almost periodically correlated
random processes with continuous time parameter}, Theory Probab.
Appl. \textbf{8} (1963), 173--177.

\bibitem{Horn1} \textsc{Horn~R.~A. and Johnson~C.~R.}, \textit{Matrix Analysis}, Cambridge University Press, 1985.

\bibitem{Hormander3} \textsc{H{\"o}rmander~L.}, \textit{The Analysis of Linear
Partial Differential Operators}, vol III, Springer-Verlag, 2007.

\bibitem{Kadison1} \textsc{Kadison~R.~V.  and Ringrose~J.~R.}, \textit{Fundamentals of the Theory of Operator Algebras}, Vol. 1--2, Academic Press, 1983, 1986.

\bibitem{Kato1} \textsc{Kato~T.}, \textit{Perturbation Theory for Linear Operators}, Springer-Verlag, 1995.

\bibitem{Levitan1} \textsc{Levitan~B.~M. and Zhikov~V.~V.}, \textit{Almost Periodic Functions and Differential Equations}, Cambridge University Press, 1982.

\bibitem{Oliaro1} \textsc{Oliaro~A., Rodino~L. and Wahlberg~P.}, \textit{Almost periodic pseudodifferential operators and Gevrey classes}, accepted for publication, Ann. Mat. Pura ed Appl., 2011. 

\bibitem{Pankov1} \textsc{Pankov~A.~A.}, \textit{Theory of almost-periodic pseudodifferential operators}, Ukrainian Math. J. \textbf{33} (5) (1981), 469--472.

\bibitem{Rabinovich1} \textsc{Rabinovich~V.~S. and Roch~S.}, \textit{Wiener algebras of operators, and applications to pseudodifferential operators}, Zeitschrift f\" ur Analysis und Ihre Anwendungen \textbf{23} (3) (2004), 437--482.

\bibitem{Reed1} \textsc{Reed~M. and Simon~B.}, \textit{Methods of Modern Mathematical Physics}, vol I,
Academic Press, 1980.

\bibitem{Rozenblum1} \textsc{Rozenblum~G.~V., Shubin~M.~A. and Solomyak~M.Z.}, \textit{Spectral Theory of Differential Operators}, Encyclopaedia of Mathematical Sciences, Vol. 64, Partial Differential Equations VII,
Springer-Verlag, 1994.

\bibitem{Rudin1} \textsc{Rudin~W.}, \textit{Functional Analysis},
McGraw--Hill, 1973.

\bibitem{Ruzhansky1} \textsc{Ruzhansky~M. and Turunen~V.},
\textit{Quantization of pseudo-differential operators on the torus}, J. Fourier Anal. Appl. \textbf{16} (2010), 943--982.

\bibitem{Ruzhansky2} \bysame,
\textit{Pseudo-differential Operators and Symmetries: Background Analysis and Advanced Topics}, Birkh\"auser, Basel, 2010.

\bibitem{Shubin1} \textsc{Shubin~M.~A.}, \textit{Differential and pseudodifferential operators in spaces of almost periodic functions}, Math. USSR-Sb. \textbf{24} (4) (1974), 547--573.

\bibitem{Shubin2} \textsc{Shubin~M.~A.}, \textit{Pseudodifferential almost-periodic operators and von Neumann algebras},
Trudy Moskov. Mat. Obshch. \textbf{35} (1976), 103--163.

\bibitem{Shubin3} \textsc{Shubin~M.~A.}, \textit{Theorems on the coincidence of the spectra of pseudodifferential
almost-periodic operators in the spaces $L^2(\rr n)$ and $B^2(\rr n)$}, Sibirian
Math. J. \textbf{17} (1976), 158--170.

\bibitem{Shubin4} \textsc{Shubin~M.~A.}, \textit{Almost periodic functions
and partial differential operators}, Russian Math. Surveys
\textbf{33} (2) (1978), 1--52.

\bibitem{Shubin4b} \textsc{Shubin~M.~A.}, \textit{The spectral theory and the index of ellipic operators with almost periodic coefficients}, Russian Math. Surveys
\textbf{34} (2) (1979), 109--157.

\bibitem{Shubin5} \textsc{Shubin~M.~A.}, \textit{Pseudodifferential Operators and Spectral Theory},
Springer, 2001.

\bibitem{Wahlberg1} \textsc{Wahlberg~P.}, \textit{A transformation of almost periodic pseudodifferential operators to
Fourier multiplier operators with operator-valued symbols}, Rend. Sem. Mat. Univ. Politec. Torino \textbf{67} (2) (2009), 247--269.

\end{thebibliography}
\end{document}